\documentclass[11pt]{amsart}

\usepackage[utf8]{inputenc}
\usepackage[T1]{fontenc}
\usepackage{lmodern}

\usepackage{amssymb,physics}
\usepackage{mathtools}
\usepackage{amscd}
\usepackage[mathscr]{eucal}
\usepackage{enumitem}
\usepackage[hmargin=1in,vmargin=1in]{geometry}

\usepackage{graphicx}
\usepackage{xcolor}

\usepackage[roman]{sublabel}

\usepackage[labelfont=bf,font=small]{caption}
\usepackage{subcaption}

\usepackage[noadjust]{cite}
\usepackage[foot]{amsaddr}

\usepackage{placeins}
\usepackage{bm}

\usepackage{varwidth}

\usepackage[ruled,vlined]{algorithm2e}
\usepackage[normalem]{ulem}
\usepackage{array}
\usepackage{tabularx}

\definecolor{CitePurple}{RGB}{128,9,158}
\definecolor{CiteBlue}{RGB}{2,95,176}
\definecolor{LinkRed}{rgb}{0.7,0,0}
\usepackage[colorlinks=true,
            citecolor=CitePurple,
            linkcolor=LinkRed,
            urlcolor=CiteBlue]{hyperref}

\SetKw{Continue}{continue}
\SetKw{False}{False}
\SetKw{True}{True}

\DeclareMathOperator{\Var}{Var}

\newcommand{\R}{\ensuremath{\mathbb{R}}}

\theoremstyle{plain}
\newtheorem{theorem}{Theorem}[section]

\newtheorem{lemma}[theorem]{Lemma}
\newtheorem{proposition}[theorem]{Proposition}

\theoremstyle{definition}
\newtheorem{definition}[theorem]{Definition}

\newtheorem{problem}[theorem]{Problem}

\newcommand{\SIADSEDIT}[1]{#1}

\begin{document}

\raggedbottom
\thispagestyle{empty}

\title{Efficient evader detection in mobile sensor networks}

\author{Henry Adams$^{1}$}
\address{$^{1}$Department of Mathematics, University of Florida, Gainesville, Florida, USA}

\author{Deepjyoti Ghosh$^{2}$}
\address{$^{2}$Department of Medicine, University of Chicago, Chicago, Illinois, USA}

\author{Clark Mask$^{3}$}
\address{$^{3}$Department of Mathematics, University of Houston, Houston, Texas, USA}

\author{William Ott$^{*,3}$}
\thanks{$^{*}$Corresponding author:
Contact \href{mailto:william.ott.math@gmail.com}{william.ott.math@gmail.com} with inquiries.}

\author{Kyle Williams$^{3}$}


\date{\today}

\begin{abstract}
Suppose one wants to monitor a domain with sensors, each sensing a small ball-shaped region, but the domain is hazardous enough that one cannot control the placement of the sensors.
A prohibitively large number of randomly placed sensors could be required to obtain static coverage.
Instead, one can use fewer sensors by providing mobile coverage, a generalization of the static setup wherein every possible evader is detected by the moving sensors in a bounded amount of time.
Here, we use topology in order to implement algorithms certifying mobile coverage that use only local data to solve the global problem.
Our algorithms do not require knowledge of the sensors' locations, only their connectivity information.
We experimentally study the statistics of mobile coverage in two dynamical scenarios.
We allow the sensors to move independently (billiard dynamics and Brownian motion), or to locally coordinate their dynamics (collective animal motion models).
Our detailed simulations show, for example, that collective motion can enhance performance:
The expected time until the mobile sensor network achieves mobile coverage is lower for the D'Orsogna collective motion model than for the billiard motion model.
Further, we show that even when the probability of static coverage is low, all possible evaders can nevertheless be detected relatively quickly by the mobile sensor network.
\end{abstract}

\maketitle

\section{Introduction}

Suppose we would like to monitor a given domain with a network of sensors.
In certain important situations, we may not be able to prescribe exact locations for the sensors.
Fallen rubble might render a domain difficult to access.
Monitoring forest fires requires navigating hazardous domains.
Further, the sensors may have limited capabilities.
We assume that each \SIADSEDIT{sensor} can detect objects only within a ball of small radius.
In such situations, placing the sensors randomly may yield \emph{static coverage}, meaning that the domain is covered at a single moment in time by the union of the sensing balls.
The statistics underlying random static coverage have been extensively studied~\cite{hall1988introduction}.
Surprisingly, we can use topological techniques to verify when a random placement of sensors produces static coverage, even if the sensor positions are not known!
Indeed, the network connectivity of the sensors alone is sufficient input~\cite{Coordinate-free,de2007coverage}.

The number of randomly placed sensors required for static coverage may be prohibitively large.
Further, environmental dynamics, such as those produced by wind or ocean currents, may cause sensors to move.
We therefore study \emph{mobile coverage}, a type of dynamic coverage that often requires significantly fewer sensors than its static counterpart.
Indeed, we consider mobile sensor networks wherein the sensors may move deterministically or stochastically.
By \emph{mobile coverage} we mean that no possible intruder could continuously move in the domain under surveillance without at some point being seen by at least one of the moving sensors.
These sensors may achieve mobile coverage even in the absence of static coverage --- i.e.\ even if the entire domain is never covered at any fixed moment in time by the union of the sensing balls.

Subtle topological tools exist that guarantee mobile coverage, even if the sensors are not able to monitor their position \SIADSEDIT{coordinates~\cite{Coordinate-free,DeSilva2007homological,de2007coverage}}.
Using these tools as a foundation, we initiate the study of the statistics of mobile coverage in minimal mobile sensor networks.
By \emph{minimal}, we mean that each sensor measures only a small amount of \SIADSEDIT{local} information, and in particular the sensors do not measure their locations.
As a consequence, we must verify mobile coverage using only local data.
This local data consists of connectivity information (overlapping sensors can detect each other), weak angular data (sensors can measure cyclic orderings of neighbors), and distance data (each sensor can measure distances to nearby detected sensors).
We address the following statistical questions.
Given randomly placed sensors which then move deterministically or stochastically in a domain, how do the statistics of the (random) \emph{mobile coverage time} $T_{\max}$, the time at which mobile coverage is achieved, behave?
Additionally, how do the mean and variance of $T_{\max}$ depend on the sensor dynamics?

We provide concrete algorithms that certify mobile coverage and use them to examine several models of sensor motion; see \cite{OurGitHubCode} for code and documentation.
These models fall into one of two categories --- either the sensors move independently, or they locally coordinate their dynamics.
We examine two models wherein the sensors move independently, one deterministic (billiard motion) and one stochastic (Brownian motion).
Inspired by the emergent dynamics associated with collective animal motion, we study the seminal D'Orsogna model~\cite{DOrsogna2006self} as well.
In this model, sensors use local attractive and repulsive forces to coordinate motion.

As an example application, suppose that deploying a sensor requires a fixed cost, and that having the sensors turned on while roving requires a marginal cost.
What is the cheapest way to achieve coverage?
Is it cheaper to try and achieve static coverage, by deploying enough randomly-placed sensors to cover the domain at the start time with high probability?
Or, is it cheaper to deploy fewer sensors, which likely will not cover the domain initially, but will provide mobile coverage after a certain duration of wandering?
How many mobile sensors do we expect to be the most cost-efficient?
Our simulations aim to allow us to answer such questions.

We survey related work in Section~\ref{sec:related}, discuss mobile coverage problems and theoretical tools in Section~\ref{sec:mobile}, show the alpha complex can be computed using local information in Section~\ref{section:alpha-local}, and introduce our algorithms in Sections~\ref{sec:algorithm} and~\ref{sec:realization}.
We examine how mobile sensor network performance statistics depend on the dynamics of the mobile sensors in Sections~\ref{sec:stochastic} and~\ref{sec:deterministic}.
We finish with open questions and concluding remarks in Section~\ref{sec:conclusion}.

\section{Related Work}
\label{sec:related}

\SIADSEDIT{We discuss related work on static coverage, on a variety of mobile sensor network models, on topological approaches to minimal sensing, on the minimal sensing problem we study in this paper, on Reeb graphs, and on further extensions to the evasion problem.}

\subsection{Static coverage}

What is the probability that immobile ball-shaped sensors cover a domain?
This is equivalent to asking what is the probability that mobile ball-shaped sensors cover a domain at time zero.
These static coverage questions have a long history in statistics and mathematics.
Sections~1.6 and 3.7 of~\cite{hall1988introduction} discuss both upper and lower bounds on the probability of coverage by $n$ sensors, and ~\cite{Schroeter1984distribution} addresses moment problems related to static coverage.
Exact formulas for the probability of coverage by a fixed number of sensors are typically unknown.
Instead, it is common to take a limit where the number of sensors goes to infinity while their radii go to zero.
Kellerer~\cite{Kellerer1983number} considers the expected number of connected components of a random sample of balls, as well as the expected Euler characteristic.
For connections between static coverage and the properties of convex hulls of multivariate samples, see~\cite{Moran1974volume}.

\subsection{A panoply of mobile sensor network models}

In this paper, we study a mobile sensor network problem.
Before we describe the problem of interest, let us survey a few of the variants on mobile sensor network problems that have been considered.

There are a wide variety of pursuit-evasion mobile sensor network problems.
Liu, Dousse, Nain, and Towsley~\cite{Liu2013dynamic} require that the evader remains in a given sensing ball for a certain amount of time before detection occurs.
They show in this context that sensors should not move too quickly if we wish to optimize mobile sensor network performance.
By contrast, in the evasion problem we will consider, we deem an evader to have been detected at the instant it enters a sensing ball.
One could study \emph{intelligent} sensor networks wherein sensor trajectories actively respond to evader dynamics.
For example,~\cite{Chin2010detection, Liu2005mobility, Liu2013dynamic} study game-theoretic approaches for the optimal evasion and detection strategies in the setting when both evaders and sensors are intelligent.
However, in the evasion problem we consider, the sensor motion and evader motion will be independent.
One could impose constraints on evader motion (such as bounded speed) and then ask what is the optimal evasion strategy, given these constraints.
We will assume that evaders can move arbitrarily fast, provided their motion remains continuous.
We refer the reader to~\cite{Chung2011search} for a taxonomy of many different of pursuit-evasion problems.

A variety of ideas have been harnessed to analyze and improve the effectiveness of mobile sensor networks.
From a control theory perspective, flocking algorithms have been combined with Kalman filters for efficient target tracking~\cite{La2009flocking, Olfati-Saber2006flocking, Olfati-Saber2007distributed, Olfati-Saber2007consensus, Olfati-Saber2012coupled, Su2016distributed, Su2017distributed}.
The collective behavior of the ant species \emph{Temnothorax albipennis} has inspired distributed coordination algorithms for tracking~\cite{Yuan2019Temnothorax}.
Neural networks can learn \emph{a priori} unknown environments while dynamic coverage is achieved~\cite{Qu2014finite}.

\subsection{\SIADSEDIT{Topological approaches to minimal sensing problems}}

In minimal sensing, one tries to use sensors measuring only local data in order to solve a global problem.
Topological techniques naturally facilitate local-to-global analysis.
See~\cite{DeSilva2007homological} for an overview of topological approaches to minimal sensing problems.

Two of the papers introducing coverage problems in sensor networks from the topological perspective are by de Silva and Ghrist.
In~\cite{Coordinate-free}, the authors consider a wide variety of sensor models (static sensors, mobile sensors) and problems (detecting static coverage, detecting mobile coverage, identifying and turning off redundant sensors).
They show how coordinate-free connectivity data, when combined with homology, can address this diverse spectrum of problems.
In~\cite{de2007coverage}, the authors incorporate persistent homology in order to obtain more refined coverage guarantees.

\subsection{\SIADSEDIT{The evasion problem}}

The mobile sensor coverage we consider in this paper is derived from Section~11 of~\cite{Coordinate-free}.
In this section, de Silva and Ghrist derive a one-sided criterion using relative homology, which can guarantee the existence of mobile coverage, but which cannot be used to show the nonexistence thereof.
In~\cite{EvasionPaths}, Adams and Carlsson develop a streaming version of this one-sided criterion using zigzag persistent homology.
Zigzag persistence can also be used for hole identification, tracking, and classification in coordinate-free mobile sensor networks~\cite{Gamble2012applied}.

In the static setting, connectivity data alone determines if the region is covered or not~\cite{Coordinate-free,de2007coverage}.
However, as Adams and Carlsson show, time-varying connectivity data alone (and hence zigzag persistence) cannot determine whether or not mobile coverage exists~\cite{EvasionPaths}.
As a result, Adams and Carlsson augment the connectivity data by considering planar sensors that also measure the cyclic orderings of their overlapping neighbors, as well as distances to these neighbors.
This is the type of minimal mobile sensor networks we study in this paper.

\subsection{\SIADSEDIT{Connections to Reeb graphs}}

One can model a mobile sensor network as two sets in spacetime: a covered region and an uncovered region.
The map to time can be thought of as a real-valued function, whose domain we often restrict to be only the uncovered region.
The \emph{Reeb graph} of the uncovered region, equipped with this real-valued map to time, encodes the space of possible (undetected) intruder motions.
From this Reeb graph, it is easy to determine whether or not an evasion path exists.
Our algorithms in this paper allow us to not only detect whether or not an evasion path exists, but also to compute the entire Reeb graph of the uncovered region.
\SIADSEDIT{Though there are a wide variety of algorithms for computing Reeb graphs~\cite{biasotti2008reeb,pascucci2007robust}, they typically focus on computing the Reeb graph of the measured space (the covered region), where in this paper we compute the Reeb graph of the complement of this space (the uncovered region).}

\subsection{\SIADSEDIT{Further extensions to the evasion problem}}

In a more theoretical direction, the papers~\cite{ghrist2017positive,carlsson2020space} give if-and-only-if criteria for the existence of an evasion path using positive homology and Morse theory, but do not address how minimal sensors would measure the inputs required for these criteria.
For example, the input required to compute the example in Section~6 of~\cite{ghrist2017positive} is enough information to instead directly compute the Reeb graph of the uncovered region.
The papers~\cite{AdamsThesis,EvasionPaths,carlsson2020space} raise and partially address the problem of computing the entire \textit{space} of evasion paths, which in~\cite{distributed} is implemented for planar sensors that can measure cyclic orderings of and local distances to nearby sensors.
Though~\cite{distributed} assumes that each sensor knows its positional coordinates, this assumption can be slightly weakened: if instead each sensor only knows which sensors it overlaps with, and the exact distances to overlapping sensors, then this is enough information to compute alpha complexes \SIADSEDIT{(see Section~\ref{section:alpha-local})}.

\section{Mobile coverage problems and theoretical tools}
\label{sec:mobile}

Let $D\subseteq \R^{2}$ be a compact domain that is homeomorphic to a closed disk.
We assume that the boundary $\partial D$ is piecewise-smooth.
Let $X\subseteq D$ be a finite set of sensor locations.
We equip the sensors with a sensing radius $r$, within which they can detect intruders and other sensors.
More precisely, each sensor $x\in X$ covers a ball $B_r(x)=\{y\in\R^{2} : \|y-x\| \leqslant r\}$.
Sensor $x$ immediately detects any intruder or sensor inside this sensing ball.
We further assume that there exists a subset of immobile \emph{fence} sensors, $F\subseteq X$, that cover the boundary of the domain.
In symbols, we require $\partial D \subseteq \cup_{x \in F} B_{r}(x)$.
We may think of these fence sensors as being on the boundary, though this is not strictly necessary.

The \emph{covered region} of the sensors is the set $C = D \cap (\cup_{x \in X} B_r(x))$.
The static coverage problem asks if $D = C$ (equivalently, if $D \setminus C = \emptyset $).
We do not need to know the coordinates of the sensors in order to address this problem.
Indeed, using only sensor connectivity data as input, de Silva and Ghrist~\cite{Coordinate-free,de2007coverage} formulate homological conditions that determine static coverage.

\subsection{Mobile coverage}

This paper focuses on mobile coverage problems, wherein sensors are allowed to move.
We modify the static setup as follows.
Let $I=[0,T]$ parametrize time, for some fixed maximum value $T>0$.
Each sensor now traces a continuous path $x \colon I\to D$.
Let $\mathscr{X}$ denote the collection of these paths.
As in the static case, we assume there exists a subset of $\mathscr{X}$ corresponding to fence sensors.
The fence sensors are immobile, meaning $x(t)=x(t')$ for all $t,t'\in I$ if $x$ is a fence sensor.
We continue to assume that the boundary $\partial D$ is covered by the sensing balls of these (static) fence sensors.

We are interested in the ability of the mobile sensor network to detect all possible continuously moving intruders.

\begin{definition}
The \emph{time-varying covered region} $C(t)$ is the union of the sensor balls at time $t$, namely $C(t)= D \cap (\cup_{x \in \mathscr{X}} B_r(x(t)))$ for $t \in I$.
An \emph{evasion path} in the mobile sensor network is a continuous path in the uncovered region, i.e.\ a continuous map $p\colon I\to D$ with $p(t)\notin C(t)$ for all $t \in I$.
We say that a mobile sensor network exhibits \emph{mobile coverage} on $D \times I$ if no evasion path exists.
\end{definition}

Note that we place no restrictions on the modulus of continuity of an evasion path (we do not assume H\"{o}lder or Lipschitz continuity, for instance).
In particular, we place no bound on the speed at which a potential evader may travel, although we do require it to move continuously.
An intruder successfully avoids detection by moving continuously and never falling into the time-varying covered region.
Our problem of mobile coverage is really a question about the existence of evasion paths.

\begin{problem}\label{p:evasion-path-problem}
Fix $T > 0$.
Given sensor paths $(x_{n})$, determine if an evasion path exists on $D \times [0,T]$.
\end{problem}

We extend Problem~\ref{p:evasion-path-problem} by not fixing $T$ in advance, but rather asking for the maximal time interval over which an evasion path exists.

\begin{problem}\label{p:max-T}
Given sensor paths $(x_{n})$, find the \emph{detection time}
\begin{equation*}
T_{\max} = \sup \{ T \geqslant 0 : \text{there exists an evasion path on } D \times [0,T] \}.
\end{equation*}
We note that $T_{\max}$ might be infinite.
\end{problem}

These mobile coverage problems motivate our two primary goals.
First, we want to develop efficient algorithms that solve these mobile coverage problems using as little information as possible about the sensor paths.
Our second goal is to understand how the statistics of this detection time depend on sensor dynamics, the number of sensors, the sensing radius, and the geometry of the domain $D$.

Suppose that each sensor can measure the cyclic ordering of its neighbors.
If the sensor network remains connected at each time and if the sensors can measure the time-varying \emph{alpha complex}, then there exists an if-and-only-if criterion for the existence of an evasion path~\cite[Theorem~3]{EvasionPaths}.
The paper~\cite{EvasionPaths} also outlines a constructive method for determining the existence of these evasion paths.
We now give the necessary background material to make this algorithm explicit.

\subsection{The alpha complex}
\label{ssec:alpha}

Fix a finite set $X$ of points in $\R^2$ (the sensor positions at a fixed time).
For any $x\in X$, the \emph{Voronoi region} $V_x\subseteq \R^2$ is defined as
\begin{equation*}
V_x=\left\{y\in \R^2 : \|y-x\| \leqslant \|y-x'\| \text{ for all }x'\in X\setminus\{x\} \right\}.
\end{equation*}
That is, the Voronoi region about $x$ contains all points in the plane that are at least as close to $x$ as they are to any other point in $X$.
For example, if $X=\{x,x'\}$ consists of two points in the plane, then each Voronoi region will be a half-space, with the perpendicular bisector of the line segment $xx'$ as the line dividing the two half-spaces.

\begin{figure}[ht]
\centering
\includegraphics[width=9cm]{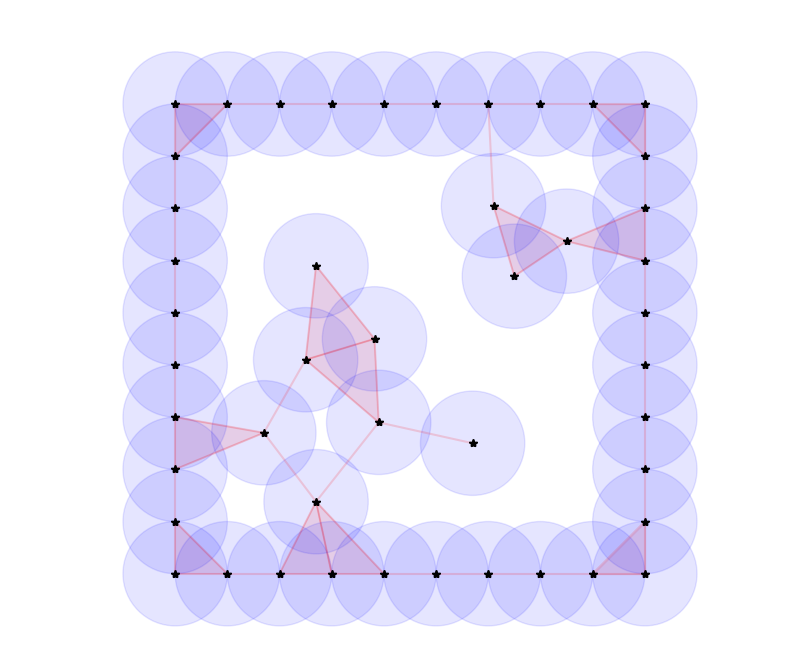}
\caption{
\textbf{Alpha complex produced by a mobile sensor network at a fixed time.}
Black stars denote sensor positions and shaded disks depict sensing balls.
Line segments and shaded triangles represent 1-simplices and 2-simplices in the alpha complex, respectively.}
\label{fig:alpha}
\end{figure}

Each sensor $x\in X$ covers a closed ball $B_r(x)$ of radius $r$.
The alpha complex of $X$ is defined as follows~\cite{EdelsbrunnerHarer,edelsbrunner1994three}.
Consider the convex regions $B_r(x)\cap V_x$ obtained by intersecting the closed ball of radius $r$ about $x$ with the Voronoi region corresponding to $x$.
The \emph{alpha complex} is a simplicial complex, with vertex set $X$, that contains a $(k-1)$-simplex if the corresponding $k$ convex regions $B_r(x)\cap V_x$ have a point of mutual intersection; see Figure~\ref{fig:alpha} for an example.
In particular, we have a vertex (a 0-simplex) in the alpha complex for each convex region $B_r(x)\cap V_x$, i.e.\ we have a vertex for each point $x\in X$.
We have an edge (a 1-simplex) $xx'$ in the alpha complex whenever two convex regions $B_r(x)\cap V_x$ and $B_r(x')\cap V_{x'}$ intersect.
When three such regions share a point of intersection, we have a triangle (a 2-simplex) in the alpha complex.
By the nerve lemma~\cite[Corollary~4G.3]{Hatcher}, the alpha complex is homotopy equivalent to the union of the convex regions $B_r(x)\cap V_{x}$, which is the same space as the union of the sensing balls $B_r(x)$ (the covered region).

\SIADSEDIT{We prove in Section~\ref{section:alpha-local} that the} alpha complex can be determined from only the pairwise distances between sensors whose corresponding sensing balls overlap.
This is local data --- we do not need to know the distances between non-overlapping sensing balls in order to compute the alpha complex.
Hence we can use alpha complexes in minimal sensing algorithms.

\subsection{\SIADSEDIT{Ribbon graphs} and boundary cycles}
\label{ssec:boundary-cycles}

A \emph{\SIADSEDIT{ribbon graph}}~\cite{igusa2002higher,mohargraphs} is a graph equipped with a cyclic ordering of the directed edges about each of its vertices.
This information is referred to as \emph{rotation data} or \emph{weak rotation data}: it is ``weak'' in the sense that exact angles are not needed.
Using the cyclic ordering, the directed edges can be partitioned into closed loops known as \emph{boundary cycles}.
These boundary cycles will be used to identify connected components of the uncovered region.
This translation from cyclic information about vertices to boundary cycles is useful in our study of mobile planar sensor networks for the following reason: cyclic orderings of neighbors can be detected by sensors with minimal capabilities because this information is local in nature.
For example, each sensor could use a rotating radar tower to order its neighbors.
By contrast, boundary cycles are global in nature, and some of these boundary cycles correspond to holes in the sensor network.
Therefore, by using the translation from rotation information to boundary cycles, we are able to convert local data measured by sensors into global data representing the connected components of the uncovered region.

\begin{figure}[ht]
\centering
\includegraphics[width=1\textwidth]{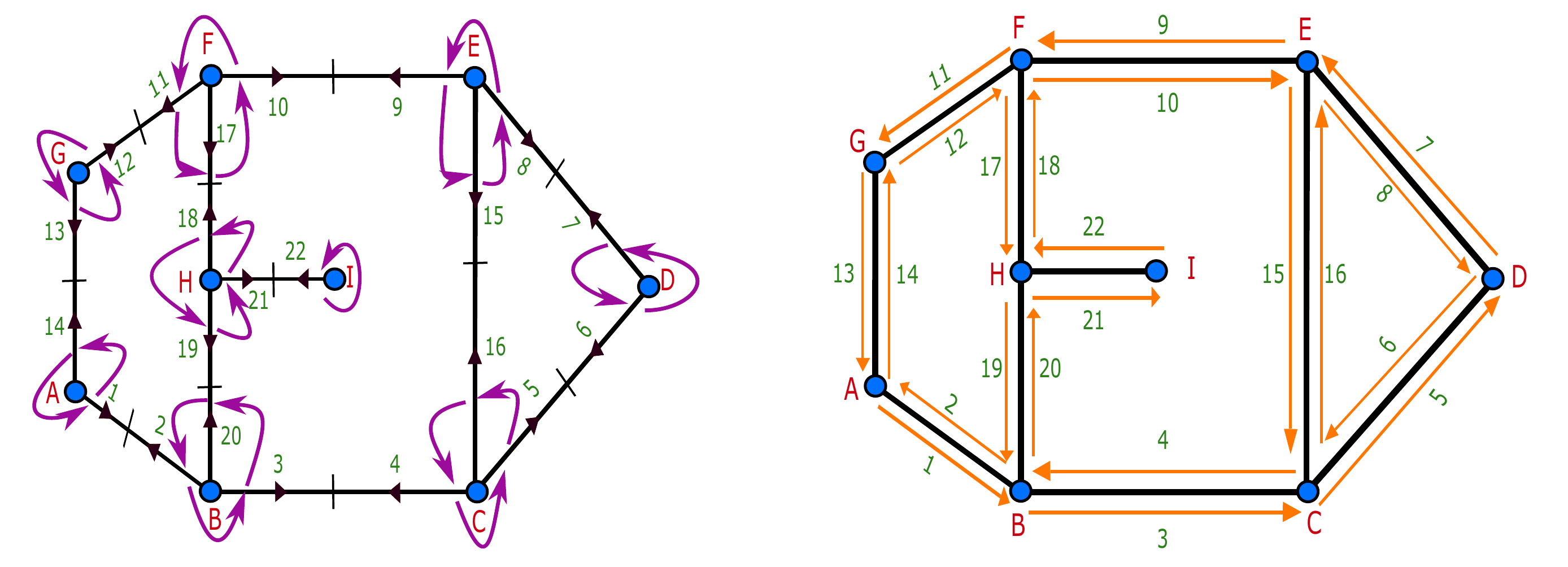}
\caption{
An example \SIADSEDIT{ribbon graph} (left) along with its boundary cycles (right).
}
\label{fig:boundary-cycles}
\end{figure}

Figure~\ref{fig:boundary-cycles} (left) depicts a sample \SIADSEDIT{ribbon graph} equipped with labels on its vertices and half-edges.\footnote{
Using the word ``half-edge'' instead of ``directed edge'' connotes that the edge will appear with both orientations.
}
Each half-edge is oriented.
For example, $1$ denotes the half-edge from vertex $A$ to $B$, whereas $2$ denotes the half-edge from $B$ to $A$.
Given a \SIADSEDIT{ribbon graph}, we have two bijections $\alpha$ and $\sigma$ to and from the set of half-edges on the \SIADSEDIT{ribbon graph}: $\alpha$ switches the two half edges corresponding to each edge, and $\sigma$ returns the next half-edge about a vertex in a counter-clockwise direction.
In Figure~\ref{fig:boundary-cycles}, for instance, we have $\alpha(1)=2$ and $\alpha(2)=1$, $\alpha(3)=4$ and $\alpha(4)=3$, et cetera.
At vertex $B$, we have $\sigma(3)=20$, $\sigma(20)=2$, and $\sigma(2)=3$.
The \emph{boundary cycles} of a \SIADSEDIT{ribbon graph} are the orbits of the composition $\varphi=\sigma \circ \alpha$.
In our example, we compute that $\varphi(2)=14$, $\varphi(14)=12$, $\varphi(12)=17$, $\varphi(17)=19$, and $\varphi(19)=2$, so one of the boundary cycles is the orbit $(2,14,12,17,19)$.
The other three boundary cycles in this example are $(4,20,21,22,18,10,15)$, $(6,16,8)$, and $(1,3,5,7,9,11,13)$; see Figure~\ref{fig:boundary-cycles} (right).

We note that each boundary cycle corresponds to a connected component of the complement of the graph~\cite{EvasionPaths,vijayan1982planarity}.
As a consequence, each boundary cycle corresponds either to a connected component of the complement of the alpha complex, or alternatively to a 2-simplex of the alpha complex.
Therefore, the boundary cycles will allow us to identify the Reeb graph of the uncovered region in spacetime.

\subsection{The Reeb graph}

Let $C(t)\subseteq D$ be the \SIADSEDIT{\textit{covered region at time} $t$}, and let its complement $U(t)= D\setminus C(t)$ be the \SIADSEDIT{\textit{uncovered region at time} $t$}.
In spacetime $D \times I$, let $\mathcal{C} = \cup_{t\in I}(C(t)\times\{t\})$ denote the \SIADSEDIT{\textit{covered region of spacetime}}, and let $\mathcal{U} = \cup_{t\in I}(U(t)\times\{t\})$ denote the \SIADSEDIT{\textit{uncovered region of spacetime}}.

\SIADSEDIT{When we use the terms covered region and uncovered region, we are referring to either a subset of space or to a subset of spacetime, depending on the implied context.}

Let $\pi \colon \mathcal{U} \to I$ be the map from the uncovered region to time, given by $\pi(x,t)=t$.
The \emph{Reeb graph} of this map $\pi$ encodes how the connected components of the uncovered region vary with time: how they split, merge, are born, and die.
The goal in many mobile evasion problems can be recast as finding the Reeb graph of $\pi$, since from the Reeb graph, we can read off all possible evasive motions moving forward in time.

\begin{figure}[ht]
\centering
\includegraphics[width=6cm]{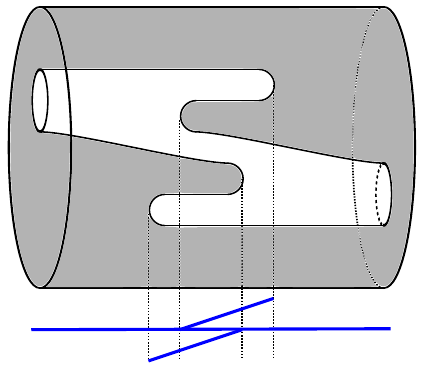}
\hspace{10mm}
\includegraphics[width=6cm]{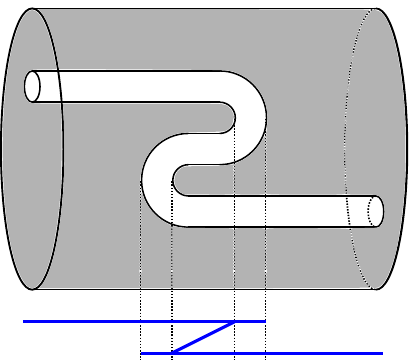}
\caption{
\textbf{Spacetime representations of two mobile sensor networks.}
Covered regions are gray and uncovered regions are white, with time increasing from left to right.
Each Reeb graph of the projection map $\pi $ from the uncovered region $\mathcal{U}$ to time $I$ appears in blue underneath the spacetime representation.
}
\label{fig:Reeb}
\end{figure}

Formally, we define the Reeb graph as follows.
Define an equivalence relation on $\mathcal{U}$ by declaring $(x,t)\sim (x',t)$ whenever $x$ and $x'$ belong to the same connected component of $\pi^{-1}(t)$.
In particular, if $t\neq t'$ then necessarily $(x,t)\not\sim(x',t')$.
The \emph{Reeb graph} of $\pi$ is the quotient space $\mathcal{U}/\sim$.
See Figure~\ref{fig:Reeb} for two examples.

In a mobile sensor network, each boundary cycle of the 1-skeleton of the alpha complex at time $t$ corresponds either to a connected component of the uncovered region $U(t)$, or alternatively to a 2-simplex of the alpha complex (i.e.\ a region of space covered by sensing domains).
In solving Problems~\ref{p:evasion-path-problem} and~\ref{p:max-T}, we will track the boundary cycles as they are born, split, evolve, merge, and die.
We will maintain labels on boundary cycles over time, encoding whether or not they may contain an intruder.
Unseen intruders can never be present in a boundary cycle corresponding to a 2-simplex, and they may or may not be present in other boundary cycles depending on the time-evolution of the network.
By tracking the boundary cycles, we will actually be tracking the connected components of $U(t)$ as time $t$ varies.
The algorithms presented in this paper are thus equivalent to computing the Reeb graph of \SIADSEDIT{$\pi\colon \mathcal{U}\to I$, using only information about the covered region $\mathcal{C}=(D\times I)\setminus \mathcal{U}$}.

\section{Local distances determine the alpha complex}
\label{section:alpha-local}

\SIADSEDIT{Since we wish to use alpha complexes in the minimal mobile sensing algorithms we develop, we must show that the alpha complex can be determined from local data.
In this section, we show that the alpha complex can be determined from only the pairwise distances between sensors whose corresponding sensing balls overlap.}

Let $X = \{x_1, x_2, \ldots, x_n \} \subseteq \mathbb{R}^m$ be a collection of $n$ points in general position.
Let $A(X,r)$ be the alpha complex on $X$ with radius parameter $r >0$, as defined in Section~\ref{ssec:alpha} for the particular case $m=2$.
The following proposition says that if one knows only the local distances between points in $X$ (those distances between points whose closed balls of radius $r$ overlap), then this knowledge is sufficient to compute the alpha complex $A(X,r)$.

\begin{proposition}
\label{prop:alpha-local}
Suppose that for each $i, j \in \{1, \ldots,  n\}$ we know
\begin{itemize}
    \item whether or not $\|x_i-x_j\| \leq 2r$, and
    \item the value of $\|x_i-x_j\|$ if $\|x_i-x_j\| \leq 2r$.
\end{itemize}
Then we can determine the alpha complex $A(X,r)$.
\end{proposition}

To prove this claim we need some definitions and a lemma from~\cite{kerber20133d} (see also~\cite{chintakunta2013distributed} for related ideas).
Let $\sigma$ be a $k$-simplex with vertices in $X \subseteq \R^m$.
Due to our general position assumption, there is a unique $(k-1)$-dimensional sphere in the $k$-plane containing $\sigma$ that circumscribes $\sigma$.
Let the center and radius of this sphere be the {\em circumcenter} and {\em circumradius} of $\sigma$.
Let the $m$-dimensional ball with this center and radius be the {\em circumball} of $\sigma$.

\begin{definition}
A simplex $\sigma$ is {\em short} if its circumradius is at most $r$.
\end{definition}

\begin{definition}
A simplex $\sigma$ is {\em Gabriel} if its circumball has no point of $X$ in its interior.
\end{definition}

\begin{lemma}[\cite{kerber20133d}]
\label{lem:short-gabriel}
A simplex $\sigma$ is in the alpha complex $A(X,r)$ if and only if it is
\begin{itemize}
    \item short and Gabriel, or
    \item the face of a simplex that is short and Gabriel.
    \end{itemize}
\end{lemma}

We would like to thank Michael Kerber for showing us the following proof.

\begin{proof}[Proof of Proposition~\ref{prop:alpha-local}]
Let $\sigma$ be a $k$-simplex with with vertices in $X$.
By Lemma~\ref{lem:short-gabriel} it suffices to determine if $\sigma$ is short and Gabriel, because $A(X,r)$ is the collection of all such simplices and their faces.

If the distance between any two vertices of $\sigma$ is more than $2r$, then $\sigma$ is not short.
Hence we assume that all pairwise distances between vertices of $\sigma$ are at most $2r$, and so we know these distances exactly.
We thus know the shape of $\sigma$ up to a rigid motion of $\R^m$.
This allows us to determine the circumcenter and circumradius of $\sigma$.
Simplex $\sigma$ is short if its circumradius is at most $r$.

Next we must determine if $\sigma$ is Gabriel.
Let $x$ be a point in $X$ that is not a vertex of $\sigma$; we must decide if $x$ is in the circumball of $\sigma$.
If the distance between $x$ and any vertex of $\sigma$ is more that $2r$, then $x$ is not in the circumball of $\sigma$ because the circumball has radius at most $r$.
Hence we assume that the distance between $x$ and each vertex of $\sigma$ is at most $2r$, and so we know these distances exactly.
We thus know the shape of $\sigma \cup \{x\}$ up to a rigid motion of $\R^m$.
This allows us to determine if $x$ is in the circumball of $\sigma$, and hence if $\sigma$ is Gabriel.
\end{proof}

\section{Algorithm for mobile coverage} 
\label{sec:algorithm}

The algorithm we present is based on~\cite[Theorem~3]{EvasionPaths}, which states that the time-varying alpha complex and time-varying rotation information (the cyclic order of edges around each vertex) are sufficient to determine whether or not an evasion path exists in a connected mobile sensor network.
While making this theorem algorithmic, one of our key contributions is to present an implementation which is not state-based, i.e.\ not based on cases.

The structures used in~\cite[Theorem~3]{EvasionPaths} are presented in an ad hoc way; what is the intuition behind using the alpha complex with rotation data?
The intuition comes from~\cite[Figure~13]{EvasionPaths}, which shows that the time-varying connectivity data\footnote{i.e., either the time-varying \v{C}ech, alpha, or Vietoris-Rips complexes} in a sensor network alone cannot determine whether or not an evasion path exists.
More information about how the covered region is embedded into spacetime is required.
Since our sensors live in the plane, this extra embedding information can take the form of weak rotation data: each sensor measures the cyclic order of its neighboring sensors.
Through the ideas presented in Section~\ref{ssec:boundary-cycles}, we are able to obtain a combinatorial representation for each of the uncovered region's connected components.
We note that this rotation information distinguishes the two examples in~\cite[Figure~13]{EvasionPaths}.
The question of whether or not the time-varying \v{C}ech complex of a sensor network, in conjunction with rotation data, is enough to determine if an evasion path exists is unknown --- see the open question on Page~13 of~\cite{EvasionPaths}.
However,~\cite[Theorem~3]{EvasionPaths} shows that the time-varying alpha complex is sufficient when equipped with rotation data.
This result relies on the fact that the 1-skeleton of an alpha complex forms a planar graph whereas the 1-skeleton of a \v{C}ech complex may not; see~\cite[Figure~14]{EvasionPaths}.

The main idea is to use the rotation data to extract the boundary cycles of this planar graph.
The boundary cycles effectively partition the domain into polygonal regions.
Some of these regions will correspond to connected components of the complement of the covered region, that is, regions with a portion that is uncovered by the sensors.
The other regions will correspond to the 2-simplices of the alpha complex.
This allows us to identify some boundary cycles with connected components of the complement of the covered region, and the others with 2-simplices.
We can then maintain labels on all boundary cycles and record whether the associated region may contain an intruder or not.
Indeed, these labels can be used to recover the Reeb graph, and thus determine whether or not an evasion path exists.

It should be noted that in~\cite[Theorem~3]{EvasionPaths} the authors assume the alpha complex remains connected for all times $t$.
We will maintain this assumption for now; see Section~\ref{sec:PowerDown} for more on relaxing this assumption.

\subsection{Discrete-time problem}
\label{sec:time-discrete}

So far, the evasion path problem, Reeb graph, and time-varying alpha complex have all been continuously varying in time.
We now discretize the problem in time so that we can track evasion paths in a time-stepping fashion.
In \cite{EvasionPaths}, the authors assume that the times at which changes in the alpha complex occur are known \emph{a priori}, and the time steps used in their algorithm are the time intervals between these changes.
If we wish to solve this problem without knowing in advance the times at which changes occur, we need to think about time stepping in a different way.
Instead of finding the exact times when the alpha complex changes, as studied in~\cite{kerber20133d}, we proceed in a more computationally efficient way.

Instead of stepping from change to change, we will split the time domain $I$ into discrete time intervals $[t_{n}, t_{n+1}]$, where $\Delta t = t_{n+1}-t_{n}$ is constant.
In doing this, we can think of the changes in the alpha complex as happening in the interior of the intervals.
If we follow the assumption in \cite{EvasionPaths} that only one of a fixed set of possible changes happens in the alpha complex at any time, then we can choose $\Delta t$ small enough so that at most one change occurs in each interval.
We will continue temporarily on the assumption that $\Delta t$ is sufficiently small such that any subinterval of $[t_{n}, t_{n+1}]$ will contain at most one combinatorial change in the alpha complex. 
\SIADSEDIT{In practice, it will generally not be possible to know how small to pick $\Delta t$; we will address this in Section~\ref{sec:adaptive-step}.}

We can now solve a discrete evasion path problem on $[t_{n}, t_{n+1}]$ using only information from time $t_{n}$ and $t_{n+1}$.
Formally, on each time interval we solve the following time-discrete problem: 
\begin{problem}
\label{p:discrete-problem}
    Let $\mathcal{H}_{n} = \{p(t_{n})\in D~|~p \text{ is an evasion path on } D \times [0, t_{n}] \}$ be the region which may contain an intruder at time $t_{n}$.
    Determine if there exists an evasion path $p \colon[t_{n}, t_{n+1}] \rightarrow D$ such that $p(t_{n}) \in \mathcal{H}_{n}$.
\end{problem}

The idea in using an alpha complex on $X$ is to partition $D$ into regions which can be uniquely identified via boundary cycles and will be labelled as feasibly having an intruder or not.
If a region is marked as having a possible intruder at times $t_{n}$ and $t_{n+1}$, then under some modest assumptions we can say there is an evasion path on $[t_{n}, t_{n+1}]$; see Section~\ref{sec:alpha-complex}.
In this manner, the outline of the algorithm is straightforward and is shown in Algorithm~\ref{alg:time-step}.
Once the labelling has been determined at time $t_{n+1}$, we can solve Problem~\ref{p:discrete-problem} by simply looking for any regions labelled as feasibly having an intruder.
An evasion path will exist on $[t_{n}, t_{n+1}]$ if any boundary cycle is marked as possibly having an intruder at time $t_{n+1}$.
In solving the Problem~\ref{p:discrete-problem}, we additionally compute all information necessary to proceed forward in time.

\begin{figure}[ht]
    \begin{algorithm}[H]\label{alg:time-step}
    \SetAlgoLined
    \SetKwInOut{Input}{Input}
    \SetKwInOut{Output}{Output}
    \Input{Alpha complex $A_{n}$, boundary cycles $B_{n}$, sensing radius $R$, boundary cycle labelling $L_{n}\colon B_{n} \rightarrow \{0, 1\}$ at $t_{n}$}
    \Output{A mapping $L_{n+1}\colon B_{n+1} \rightarrow \{0, 1\}$, indicating whether each boundary cycle at $t_{n+1}$ may contain an intruder}
    \Begin{
        $X_{n+1} \leftarrow UpdateSensors()$\;
        $A_{n+1} \leftarrow ComputeAlphaComplex(X_{n+1}, R)$\; 
        $rot\_data \leftarrow ComputeRotationalData(X_{n+1})$\;
        $B_{n+1} \leftarrow ComputeBoundaryCycles(A_{n+1}, rot\_data)$\;
        $L_{n+1} \leftarrow UpdateBoundaryCycleLabelling(B_{n}, B_{n+1}, A_{n}, A_{n+1}, L_{n})$\;
        \Return $L_{n+1}$
    }
    \caption{Single Time Step}
    \end{algorithm}
\end{figure}

Solving this problem on successive time intervals allows us to determine the mobile coverage, and existence of evasion paths, independently from the time domain used.
In particular, the end time and times at which changes occur in the alpha complex need not be known \emph{a priori}.
Indeed, the following lemma can be seen via a simple inductive argument:
\begin{lemma}\label{lemma:multistep}
    Solving Problem~\ref{p:discrete-problem} iteratively produces a solution to the continuous evasion path Problem~\ref{p:evasion-path-problem} and the time-discrete analog of Problem~\ref{p:max-T}.
\end{lemma}

\subsection{Combinatorial changes in the alpha complex}
\label{sec:alpha-complex}
Since the alpha complex is homotopy equivalent to the union of sensor balls, the alpha complex can determine the statically covered region of a sensing network (for further details see~\cite{Coordinate-free,de2007coverage}, which instead use the closely related Vietoris-Rips complex).
This allows us to obtain a combinatorial representation of the sensor network's covered region, allowing us to fully discretize the evasion path problem.
One feature of using the alpha complex on $X$ is that because its edges do not cross, it is easy to construct a \SIADSEDIT{ribbon graph} from it.

Now that we have discretized the problem temporally, we may talk about a discrete sequence of static alpha complexes, each giving a combinatorial representation of the topology of the covered region at time $t_{n}$.
We will denote $A_{n}$ as the alpha complex on $X$ with radius $r$ at time $t_{n}$.
We will further denote the sets of 1- and 2-simplices of $A_{n}$ as $A^{1}_{n}$ and $A_{n}^{2}$, respectively.

The ideas presented in~\cite{EvasionPaths} require that only one of a fixed list of \emph{atomic} combinatorial changes in the alpha complex occurs at each point in time.
The list of allowed atomic changes can be found in Figure~\ref{fig:possible-changes}.
In order to allow more complex combinatorial changes to happen at a time, the ideas in~\cite{edelsbrunner1990simulation} may be helpful.
\SIADSEDIT{In the discrete case, we require that at most one change happens between $t_{n}$ and $t_{n+1}$; we will assume this for the remainder of Section~\ref{sec:algorithm}.
This can be achieved for instance by assuming that $\Delta t$ is sufficiently small so that all atomic changes are sufficiently resolved.
In this context, $A_{n}$ and $A_{n+1}$ will differ by at most one of these changes.
}

\begin{figure}[ht]
    \centering
    \includegraphics[width=6.5in]{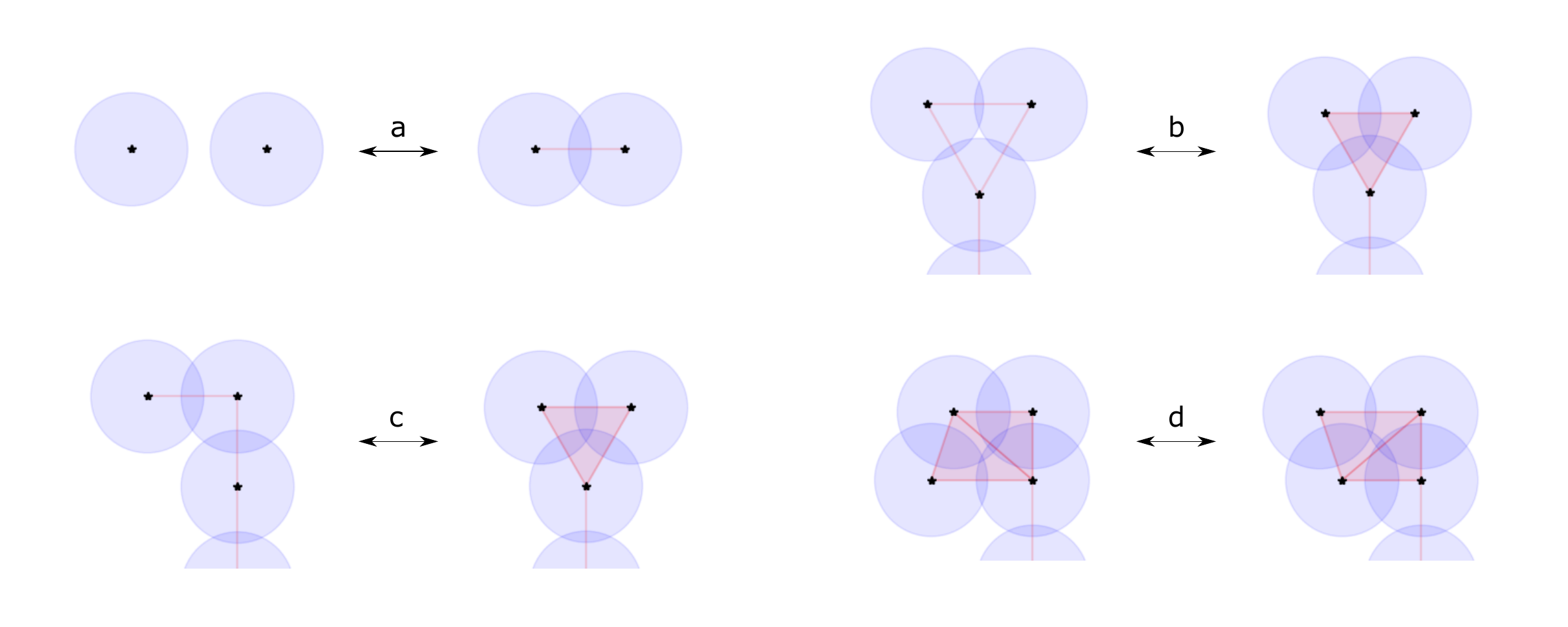}
    \caption{
        \textbf{Atomic combinatorial changes within the alpha complex.}
        \textbf{(a)}
        A single edge appears or disappears.
        \textbf{(b)}
        A single 2-simplex appears or disappears.
        \textbf{(c)}
        A 2-simplex and a free face edge simultaneously appear or disappear.
        \textbf{(d)}
        A Delaunay edge flip occurs.
    }
    \label{fig:possible-changes}
\end{figure}

\subsection{Tracking changes in boundary cycle labels}
\label{sec:boundary-cycles}

Since we are solving this problem in a \SIADSEDIT{streaming} fashion, we do not know which of the allowed combinatorial changes will occur in the alpha complex.
We simply compute $A_{n+1}$ and the associated boundary cycles and must infer by comparison with the previous time step how the boundary cycle labelling will change.
Given that a combinatorial change occurred, we wish to (1) identify which type of combinatorial change occurred, (2) analyze the changes it will induce in the boundary cycles, and (3) determine how the associated boundary cycle labelling will change.

Suppose we know the alpha complexes at time $t_{n}$ and $t_{n+1}$, denoted $A_{n}$ and $A_{n+1}$.
Additionally, assume the set of boundary cycles, $B_{n}$, are known at time $t_{n}$ with the associated labelling $L_{n}\colon B_{n} \rightarrow \{0, 1\}$.
We let $0$ denote that the corresponding cycle cannot contain an intruder, and $1$ denote that the corresponding cycle may contain an intruder.
From $A_{n+1}$ we can compute the rotation data and obtain the associated boundary cycles $B_{n+1}$.
Now with this information as input, we want to determine which combinatorial change occurred, understand the associated changes in boundary cycles, and produce the appropriate boundary cycle labelling $L_{n+1} \colon B_{n+1}\rightarrow \{0,1\}$.

We can determine the transition that occurred by counting the number of boundary cycles, 1-simplices, and 2-simplices that are new and those that are no longer present.
For example, when an edge is added, this can be uniquely detected by verifying that 
\begin{align*}
    \vert A_{n+1}^{1}\setminus A_{n}^{1} \vert = 1, &\quad \vert A^{1}_{n}\setminus A^{1}_{n+1} \vert = 0, \\ 
    \vert A_{n+1}^{2}\setminus A_{n}^{2} \vert = 0, &\quad \vert A^{2}_{n}\setminus A^{2}_{n+1} \vert = 0, \\ 
    \vert B_{n+1}    \setminus B_{n}     \vert = 2, &\quad \vert B_{n}    \setminus B_{n+1}     \vert = 1
\end{align*}
Since the network is connected, this new edge necessarily splits a boundary at time $n$ into two new ones at time $n+1$.
A table for each atomic transition is given in Table~\ref{tab:atomic}.

\begin{table}[ht]
    \centering
    \resizebox{\columnwidth}{!}{
    \begin{tabular}{|m{14em}||c|c|c|c|c|c|}
    \hline
         & \SIADSEDIT{Edges ($+$)} & \SIADSEDIT{Edges ($-$)} & \SIADSEDIT{2-simplices ($+$)} & \SIADSEDIT{2-simplices ($-$)} & \SIADSEDIT{Boundary cycles ($+$)} & \SIADSEDIT{Boundary cycles ($-$)} \\
         \hline
         \hline
         Add edge & 1 & 0 & 0 & 0 & 2 & 1 \\
         \hline
         Remove edge & 0 & 1 & 0 & 0 & 1 & 2\\
         \hline
         Add 2-simplex & 0 & 0 & 1 & 0 & 0 & 0 \\
         \hline
         Remove 2-simplex & 0 & 0 & 0 & 1 & 0 & 0 \\
         \hline
         Add free edge and 2-simplex & 1 & 0 & 1 & 0 & 2 & 1 \\
         \hline
         Remove free edge and 2-simplex & 0 & 1 & 0 & 1 & 1 & 2\\
         \hline
         Delaunay edge flip & 1 & 1 & 2 & 2 & 2 & 2\\
         \hline
    \end{tabular}
    }
    \caption{
    The number of simplices and boundary cycles added \SIADSEDIT{($+$)} or removed \SIADSEDIT{($-$)} under each atomic transition.
    }\label{tab:atomic}
\end{table}

We now analyze how each atomic transition affects the boundary cycles and their labelling; see Figures~\ref{fig:possible-changes} and~\ref{fig:boundary-cycles-changes} for illustrations of each atomic transition.
\begin{enumerate}[leftmargin=*, labelindent=\parindent]
    \item[(a+)] {\bf 1-simplex is added}.
        When a 1-simplex is added, there will always be one old boundary cycle being split into two new boundary cycles.
        So the labelling of the two new boundary cycles will inherit the label of the old (removed) boundary cycle.
        
    \item[(a$-$)] {\bf 1-simplex is removed}.
        When a 1-simplex is removed, we will always have two old boundary cycles joining together into one new boundary cycle.
        So the new boundary cycle will be labelled as potentially having an intruder if either of the old boundary cycles may have had an intruder.
    
    \item[(b+)] {\bf 2-simplex added}.
        When a 2-simplex is added, once the corresponding boundary cycle is identified, it can be set as having no intruder.
    
    \item[(b$-$)] {\bf 2-simplex removed}.
        When a 2-simplex is removed, the corresponding boundary cycle, which was previously labelled with no intruder, will continue to maintain the same label.
        
    \item[(c+)] {\bf Pair consisting of a 2-simplex and a free edge is added}.
        The label updating is similar to adding a 1-simplex, but now we must determine which boundary cycle corresponds to the new 2-simplex and set its label to have no intruder.
        This update can be achieved through the composition of (a+) adding a 1-simplex, and then (b+) adding the 2-simplex.
        
    \item[(c$-$)] {\bf Pair consisting of a 2-simplex and a free edge is removed}.
        Just as in adding a simplex pair, the labelling for this case can be done by (a$-$) removing the 1-simplex and then (b$-$) removing the 2-simplex.
        
    \item[(d)] {\bf Delaunay edge flip}.
        A \SIADSEDIT{Delaunay} edge flip occurs when the two triangles filling in a quadrilateral flip.
        The new boundary cycles will be labelled false since they are filled-in 2-simplices.
\end{enumerate}

We remark that a Delaunay edge flip occurs only in the setting when four sensor locations lie on the boundary of a single circle, with no sensors in the interior of that circle.
Though this \SIADSEDIT{occurrence} is not generic for static sensors, it will happen for mobile sensors.
The Voronoi regions of these four locations have a common intersection point at the center of this circumcircle.
The Delaunay edge flip swaps which diagonal edge is present.
The presence of such a diagonal edge in the alpha complex implies that the sensing radius is at least as large as the circumradius of these four sensor locations.
This implies that the old diagonal edge (prior to the Delaunay edge flip) is adjacent to two 2-simplices in the old alpha complex, and also that the new diagonal edge (after the Delaunay edge flip) is adjacent to two 2-simplices in the new alpha complex.

\SIADSEDIT{Each of these atomic changes can also be interpreted in terms of how it affects the Reeb graph.
The case (a+) corresponds to a boundary cycle (and thus a connected component of $U(t)$) being split into two boundary cycles, and can be associated with a critical point in the Reeb graph. 
Case (a$-$) likewise corresponds to two boundary cycles merging together, again corresponding to a critical point in the Reeb graph. 
Cases (b+) and (b$-$) imply the death and birth of a boundary cycle respectively.
Cases (c+) and (c$-$) do not correspond to a critical point in the Reeb graph since they do not create or destroy any connected components of the complement. However, they do change the representation of the boundary cycles, thus we can associate the old boundary cycle representation with the new boundary cycle representation as belonging to the same connected component of $U(t)$. 
Finally, case (d) does not affect the connected components of $U(t)$, and does not correspond to a critical point in the Reeb graph. 
Note that the assumption that these atomic changes are the only changes implies that the Reeb graph is a binary graph.
}

\SIADSEDIT{
One may explicitly construct a discrete representation of the Reeb graph by associating the times at which each atomic change occurs with a graph node.
For each boundary cycle, a directed edge is then inserted between the graph nodes beginning with the time the boundary cycle was produced and terminating at the time the boundary cycle is next modified. 
}

\begin{figure}[ht]
    \centering
    \includegraphics[width=1\textwidth]{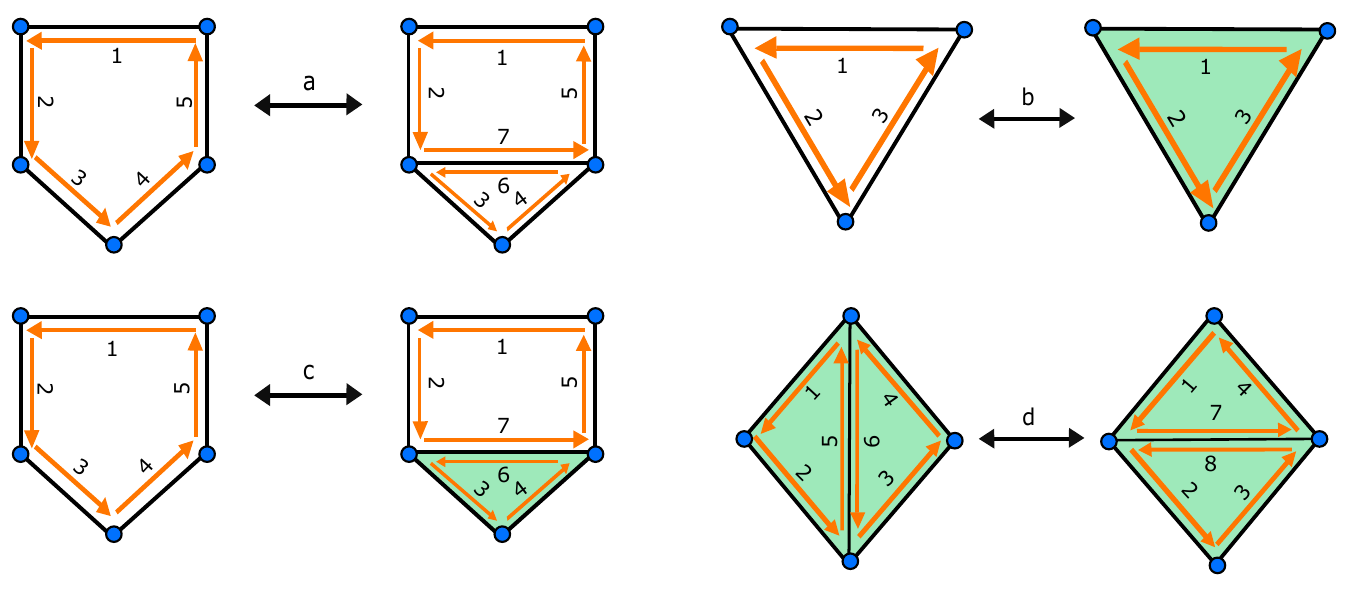}
    \caption{
    \textbf{Atomic boundary cycle updates.}
    These updates correspond to the atomic combinatorial changes within the alpha complex illustrated in Figure~\ref{fig:possible-changes}.
    }
    \label{fig:boundary-cycles-changes}
\end{figure}

\subsection{Non-state-based algorithm}
\label{ssec:non-state}
In~\cite{EvasionPaths, distributed}, the boundary cycle label update is treated as a state-based algorithm, updating the labeling depending on the case (a+), (a$-$), (b+), (b$-$), (c+), (c$-$), or (d) of which atomic transition occurred.
We observe here that, given all the current assumptions, this is not necessary.
This was initially surprising to us, given that the proof of~\cite[Theorem~3]{EvasionPaths} is heavily reliant on its case-based organization.
Indeed, this non-state-based update comes from the realization that in each atomic transition we can first add/remove a 1-simplex (if applicable) and then label all 2-simplices as having no intruder while leaving all other labels unchanged.
This is put into pseudocode in Algorithm~\ref{alg:UpdateLabel-v2}, where $\bigvee$ denotes the ``or'' operation.
We emphasize that Algorithm~\ref{alg:UpdateLabel-v2} is not divided into cases (which for purposes of comparison, would produce the much longer Algorithm~\ref{alg:UpdateLabel}).

\begin{figure}[ht]
\begin{algorithm}[H]\label{alg:UpdateLabel-v2}
\SetAlgoLined
\SetKwInOut{Input}{Input}
\SetKwInOut{Output}{Output}
\Input{Alpha complex $A_{n+1}$, the boundary cycles $B_{n}$ and $B_{n+1}$, boundary cycle labelling $L_{n}\colon B_{n} \rightarrow \{0, 1\}$ }
\Output{The mapping $L_{n+1}\colon B_{n+1} \rightarrow \{0, 1\}$, indicating whether each boundary cycle may contain an intruder}

\Begin{
    \For{$b \in B_{n}\cap B_{n+1}$}{
        $L_{n+1}[b] \leftarrow$ $L_{n}[b]$\;
    }
    \For{$b \in B_{n+1}\setminus B_{n}$}{
        $L_{n+1}[b] \leftarrow \bigvee\limits_{x \in  B_{n}\setminus B_{n+1}} L_{n}[x]$\;
    }
    \For{$b \in B_{n+1}$}{
        \If{Nodes of $b$ form a 2-simplex of $A_{n+1}$}{
            $L_{n+1}[b] \leftarrow \mathbf{False}$\;
        }
    }

    \Return $L_{n+1}$\;
}
\caption{Labelling Update - Serialized}
\end{algorithm}
\end{figure}

\subsection{Where is the fence used?}
\SIADSEDIT{
The astute reader (or an anonymous referee who we would like to acknowledge) may notice that the fence sensors do not play a larger role than any other sensors in the algorithm above.
In some sense, the fence sensors are not needed.
Suppose that the sensors and intruders live in $\R^2$.
With only a finite number of sensors with finite radii, there will always be a single unbounded connected component of the uncovered region.
If this unbounded component contains intruders, then it can never be swept by a finite number of sensors with finite radii.
However, suppose that at the initial time, only \emph{some} of the connected components of the uncovered region contain intruders, and the unbounded component of the uncovered region does not.
One could then run our algorithm, maintaining labels to denote which components could contain intruders at each time, to decide in an if-and-only-if fashion whether an evasion path exists or not.
As remarked above, if the unbounded component ever merges with one that contains intruders, then the unbounded component will forevermore contain intruders --- but that might never happen.
(Alternatively, one could also run a version of our algorithm for sensors on the $2$-sphere.)
}

\section{Realization of the evasion path algorithm}
\label{sec:realization}

There are some technicalities that must be overcome if one is to implement the algorithm described in Section~\ref{sec:algorithm}.
Specifically, we have found the assumptions of being able to choose a small enough uniform time step size, as well as maintaining connectivity of the alpha complex, quite strong assumptions if one wishes to use arbitrary models of motion.
In this section, we relax these assumptions and provide a way to solve the evasion path problem without changing the nature of the original algorithm.

\subsection{Adaptive time stepping} \label{sec:adaptive-step}

As mentioned above in Section~\ref{sec:time-discrete}, we have been assuming that our time step $\Delta t$ is sufficiently small to resolve all combinatorial changes in the alpha complex.
If we wish to select $\Delta t$ as a parameter, we have the problem that there is also the underlying inability to know exactly how small is ``sufficiently small.''
Such knowledge would require that we know at what exact times the alpha complex will undergo atomic changes, which may be possible for simple trajectories~\cite{kerber20133d}, and impractical in more complex situations. 
We found that for this assumption to hold, $\Delta t$ must generally be quite small.
With the alpha complex being expensive to compute, this renders time stepping using a uniform $\Delta t$ impractical.

Our solution to this problem is to use an adaptive time-stepping scheme to resolve interval steps of size $\Delta t$ into subintervals of unequal size, each of which contain at most one atomic change.
We recall that atomic changes are assumed to happen in the interior of our time intervals.
Note that we still require the assumption that at each point in time only one atomic transition occurs.
With this in mind, if $\Delta t$ is not small enough, two or more atomic changes may happen within a single interval (though at different times).
If we detect that $A_{n}$ and $A_{n+1}$ differ by more than an atomic change, we can recursively divide the interval in two and solve Problem~\ref{p:discrete-problem} on the subintervals $[t_{n}, t_{n}+ \frac{\Delta t}{2}]$ and $[t_{n} + \frac{\Delta t}{2}, t_{n}+\Delta t]$.
This is equivalent to performing a binary search for the intervals which differ by only one atomic change.
Pseudocode for this process is shown in Algorithm~\ref{alg:intruder-detection-adaptive} in the appendix.
This algorithm works because Lemma~\ref{lemma:multistep} provides that if there is an evasion path on $[t_{n}, t_{n}+ \frac{\Delta t}{2}]$ and on $[t_{n} + \frac{\Delta t}{2}, t_{n}+\Delta t]$ (overlapping on the same connected component of the uncovered region at time $t_{n} + \frac{\Delta t}{2}$), then there is an evasion path on $[t_{n}, t_{n}+\Delta t]$.

\SIADSEDIT{Note that if $A_{n}$ and $A_{n+1}$ are identical, this does not imply the same for every subinterval of $[t_{n}, t_{n+1}]$. 
For example, removing an edge and then adding the same edge back in within a single time interval could produce identical alpha complexes $A_{n}$ and $A_{n+1}$.
Similarly, if $A_{n}$ and $A_{n+1}$ differ by at most one atomic change, this does not imply the same for every subinterval of $[t_{n}, t_{n+1}]$.
For example, if two edges were added and one of the two then removed within a single time interval, $A_{n}$ and $A_{n+1}$ would differ by a single added edge.
Since our adaptive time-stepping scheme produces a sequence of alpha complexes such that each consecutive pair is identical or differs by a single atomic change, we assume that if a time step contains at most one atomic change, further refinement will not produce any other atomic changes.}

\SIADSEDIT{This adaptive time-stepping approach has been developed to minimize the number of alpha complex computations needed to resolve a given interval. 
Observe that when attempting to resolve $A_{n}$ and $A_{n+1}$, the alpha complex $A_{n+1}$ can be stored for later reuse when resolving $A_{n+1 - 2^{-i}}$ and $A_{n+1}$ and so on, recursively. 
The most common approach to computing the alpha complex is to compute the Delaunay triangulation and then filter out unneeded simplices.
This approach suffers from the high computational cost of the Delaunay triangulation, namely a complexity of $\mathcal{O}(K^{\lceil\frac{d}{2}\rceil +1})$, where $K$ is the number of vertices in the alpha complex~\cite{Elshakhs2024Delaunay}. 
If the rotation information can be queried from the sensors directly or maintained dynamically, then this does not add to the complexity. 
Generation of the boundary cycles has complexity $\mathcal{O}(E)$, where $E$ is the number of edges in the alpha complex, since we are simply generating the orbits of $\varphi$. 
Thus each recursion level has complexity $\mathcal{O}(K^2 + E)$, so the overall algorithm has complexity $\mathcal{O} (M (K^{2} + E))$, where $M$ is the number of steps produced by the recursion process.}

\subsection{Extension to disconnected sensor networks}
\label{sec:PowerDown}

In Section~\ref{sec:boundary-cycles}, we operate entirely under the assumption that the sensor network, or equivalently the alpha complex, remains connected.
This is a rather strong assumption if we wish to investigate arbitrary models of motion.
We need to allow the sensor network to become disconnected.

The algorithm from Section~\ref{sec:algorithm}, however, does not work in the context of disconnected sensor networks.
As can be seen from Figure~16 of~\cite{EvasionPaths}, a number of problems arise when we attempt to allow for a disconnected alpha complex.
Indeed, this figure shows two sensor networks --- both of which become disconnected --- which have the same alpha complex and rotation information at each time, yet one contains an evasion path while the other does not.
Allowing for a disconnected alpha complex creates the following problems, among others: (1) The embedding of the alpha complex in $\mathbb{R}^{2}$ may be ambiguous; (2) the identification of boundary cycles with connected components of the uncovered region is not injective; (3) the boundary cycles must be aware of their embedding relative to each other.
Even worse, if we want to fix issue (3), local information is not sufficient; nonlocal information is required.

We now present a modification of the evasion path problem that allows us to determine existence of evasion paths without this rather restrictive assumption of the sensors maintaining connectivity.
The main idea is to ignore sensors that become disconnected from the fence component until they are reconnected, in a similar model to that considered in \cite{Coordinate-free}.
This can be physically motivated by considering the scenario where disconnected sensors are out of range of the main computing network, i.e., we can imagine a setup where sensors that are disconnected from the fence are disconnected from a power source.
For this reason we will refer to this model of allowing disconnected sensors as the ``power-down model.''
Analogously, any algorithm which solves Problem~\ref{p:evasion-path-problem} and allows disconnected sensors to remain turned on will be referred to as a ``power-on model.''
We do not consider a power-on model here, and instead save it for future work, since such a model requires nonlocal information and would change the nature of the problem being solved.

Topologically, the power-down model changes the definition of the covered region as follows.
Recall that $F\subseteq X$ is the subset of fence sensors.
At each time $t$, we partition the covered region $C(t)=\cup_{x\in X}B_r(x(t))$ into two regions: the part that is in the same connected component as the fence sensors, and the part that is not in this connected component.
This partitions the set $X$ of sensors into two groups: $X=X'\coprod X''$, where $X'\supseteq F$ is the set of sensors in the same connected component of $C(t)$ as the fence sensors, and where $X''=X\setminus X'$ is the set of sensors which are not in this component.
In the power-down model, all of the sensors in $X''$ are turned off, and do not contribute to the covered region.
Indeed, the covered region in the power-down model is $\tilde{C}(t)=\cup_{x\in X'(t)} B_r(x(t))$.
We are still interested in whether or not there exists an evasion path, i.e.\ a continuous path $p\colon I\to D$ with $p(t)\notin\tilde{C}(t)$ for all times $t$.
With this modified notion of an evasion path, we can proceed with solving Problem~\ref{p:evasion-path-problem}.

We remark that our algorithm for detecting mobile coverage with possibly disconnected sensor networks in the power-down model is still local, relying only on alpha complexes (computed locally as in Section~\ref{section:alpha-local}) and weak rotation data.
Moreover, our local algorithm for the power-down model still gives one-sided bounds for a would be power-on model: if no evasion path exists with the power-down sensors, then no evasion path exists for power-on sensors undergoing the same motions.
Similarly, if an evasion path $p$ exists for the power-on model, then $p$ is also an evasion path for the power-down model.
These are consequences of the inclusion $\tilde{C}(t) \subseteq C(t)$ of the covered regions in the two different models, for all $t$.

An attractive feature of the power-down modification is that it adds no complexity or data structures to existing algorithms.
In the context of the previous sections, the only modification is that we allow for two additional atomic combinatorial changes corresponding to components becoming connected or disconnected from the fence.
We additionally can ignore all transitions that occur within a disconnected component.
We follow the same time discretization process as in Section~\ref{sec:algorithm}; the only difference is that we must handle the additional cases where components of the alpha complex become disconnected or reconnected.
This is a only a minor modification thanks to the easily seen observation:

\begin{lemma}
Under the assumption that at most one combinatorial change listed in Section~\ref{sec:boundary-cycles} happens within each time, then a component of the alpha complex can only become disconnected (or reconnected) if a single edge is removed (or added) to the alpha complex.
\end{lemma}

\subsection{Boundary cycle label updates}
When a component of the alpha complex becomes disconnected from the fence, we will have a single boundary cycle initially surrounding a hole which will then be split into two boundary cycles; see Figure~\ref{fig:disconnection} for illustration.
One of these boundary cycles will correspond to the outside of a new disconnected component; it will not be given any label.
The other boundary cycle will be a hole enclosing the disconnected component, and if it is connected to the fence then we will maintain a label for this boundary cycle, and its initial label will be the same as the prior boundary cycle that split.
In later time steps, any boundary cycle on a disconnected component will not be labeled.
This represents the idea that the sensors that become disconnected are not a part of the power-down alpha complex.
This type of atomic change can be uniquely identified just as in Section~\ref{sec:boundary-cycles}; see Table~\ref{tab:disconnected}.

When a disconnected component becomes reattached to the fence through the addition of a 1-simplex, we will have the two previously described boundary cycles merging into one.
All of the newly connected boundary cycles will be labeled to match the outer enclosing boundary cycle; however, any  boundary cycle filled by a 2-simplex will be labeled as having no intruder.
Again, this atomic change can be detected by simply counting the number of boundary cycles and simplices in the alpha complex before and after the transition; see Table~\ref{tab:disconnected}.

We can detect if a component has become disconnected or reconnected by simply extending Table~\ref{tab:atomic} with the following two cases.
\begin{table}[ht]
    \centering
    \resizebox{\columnwidth}{!}{
    \begin{tabular}{|m{14em}||c|c|c|c|c|c|}
        \hline
        & \SIADSEDIT{Edges ($+$)} & \SIADSEDIT{Edges ($-$)} & \SIADSEDIT{2-simplices ($+$)} & \SIADSEDIT{2-simplices ($-$)} & \SIADSEDIT{Boundary cycles ($+$)} & \SIADSEDIT{Boundary cycles ($-$)} \\
        \hline
        \hline
        Disconnect & 0 & 1 & 0 & 0 & 2 & 1 \\
        \hline
        Reconnect & 1 & 0 & 0 & 0 & 1 & 2\\
        \hline
        \hline
    \end{tabular}}
    \caption{
        The number of simplices and boundary cycles added \SIADSEDIT{($+$)} or removed \SIADSEDIT{($-$)} under the additional possible atomic transitions (extending Table~\ref{tab:atomic}) in the power-down model.
    }
    \label{tab:disconnected}
\end{table}

\begin{figure}[ht]
\centering
\includegraphics[width=16cm]{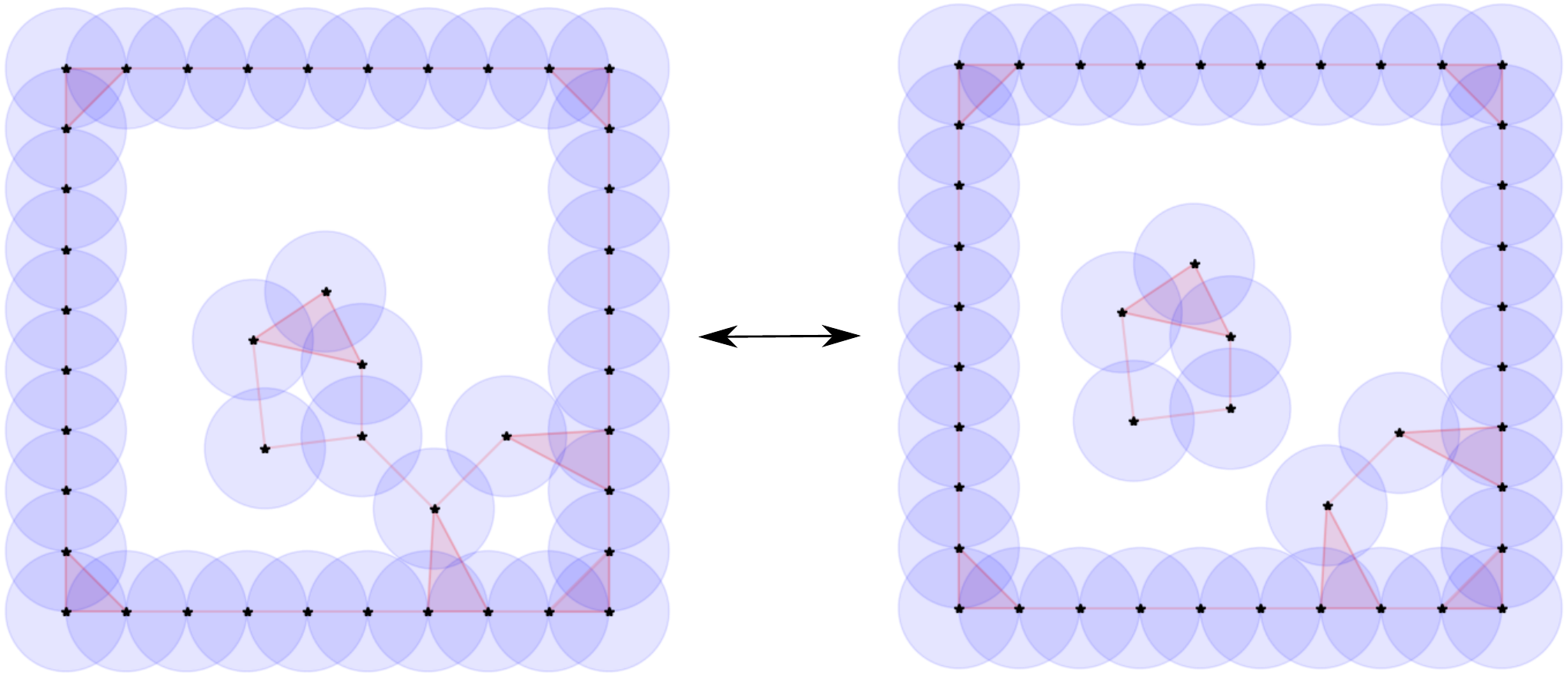}
\caption{
\textbf{A component of the alpha complex disconnects from the fence component.}
From left to right, a single edge disappears, producing an additional connected component.
From right to left, a single edge appears, merging two connected components into one.
}
\label{fig:disconnection}
\end{figure}

\section{Detection time statistics for diffusive sensors}
\label{sec:stochastic}

In this section, we compute and examine the detection time $T_{\max}$ distributions assuming the mobile sensors move diffusively.
Indeed, diffusion is relevant when minimal sensor networks experience random environmental forces, such as winds.

In all of the models of sensor motion we consider (diffusion, billiards, and collective motion) in this section and the next, the initial locations of the mobile sensors at time zero are chosen uniformly at random inside the domain.
We note that $T_{\max}$ is a random variable that can attain the value infinity.
(This happens if the sensor network never provides mobile coverage.)
Experimentally, we appear to be in regimes where $T_{\max}=\infty$ happens with probability zero, and the expected value $E[T_{\max}]$ is finite.
If needed (say for computational purposes), one could replace $T_{\max}$ with a ``capped value'' that is the minimum of $T_{\max}$ and some large positive constant.
But this was not needed for our experiments---all of our trials yielded mobile coverage in finite time.

\subsection{Domain geometry}
\label{subsec:squares}

Our domain is a scaled square $D = (1 + \delta) S$, where $S$ is the unit square centered at the origin and $\delta > 0$ is small and fixed.
The sensing domains of the immobile fence sensors cover the boundary $\partial D$.
We require that the mobile sensors remain within the unit square $S$.
In particular, mobile sensors reflect elastically off of the boundary of $S$.
We initialize the locations of the mobile sensors by placing them uniformly at random within $S$.

\subsection{Performance of Brownian mobile sensor networks}

We assume that each mobile sensor moves independently and satisfies the stochastic differential equation
\begin{equation*}
\mathrm{d} \bm{x}_{t} = f (\bm{x}_{t}) \, \mathrm{d} t + \sigma \, \mathrm{d} \bm{W}_{t},
\end{equation*}
where $f$ denotes the drift and $\bm{W}_{t}$ denotes standard two-dimensional Brownian motion.
We examine $T_{\max}$ statistics assuming no drift ($f$ vanishes), upon varying both the sensing radius $r$ and the number of mobile sensors $N$.
We generated $T_{\max}$ distributions by performing 500 mobile sensor network simulations for each $(N,r)$ pair and using our algorithms to compute $T_{\max}$ for each simulation.
We used diffusion coefficient $\sigma = 0.5$ for every simulation.
See Figure~\ref{fig:initial_density} for visualization of a sample sensor network initialization for various $(N,r)$ pairs.

\begin{figure}[ht]
\centering
\includegraphics[width=15cm]{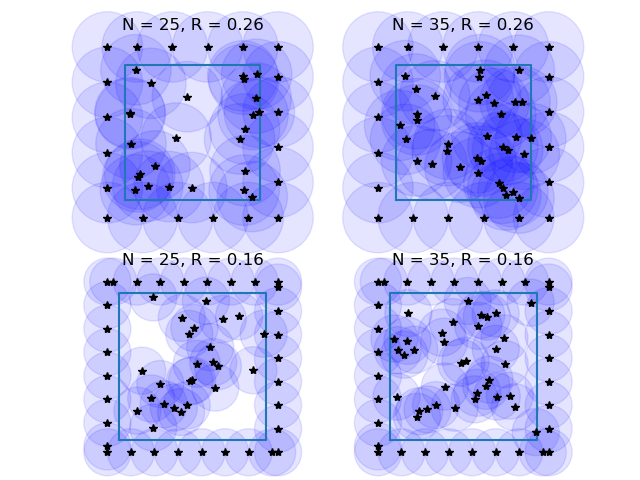}
\caption{
Sample sensor network initialization for various $(N,r)$ pairs.
}
\label{fig:initial_density}
\end{figure}

\begin{figure}[ht]
\centering
\includegraphics[width=15cm]{browninan_ex_revised.pdf}
\caption{
\textbf{Detection time statistics for Brownian mobile sensor networks.}
The heatmaps illustrate \textbf{(a)} mean detection time $E[T_{\max}]$, \textbf{(b)} detection time variance $\Var [T_{\max}]$, and \textbf{(c)} the probability that the union of the sensing balls covers the entire domain $D$ at time $t=0$.
We vary the sensing radius $r$ and the number $N$ of mobile sensors.
\textbf{(d)}
Detection time distributions (smoothed histograms) for various $(N,r)$ pairs.
}
\label{fig:Brownian-results}
\end{figure}

Figure~\ref{fig:Brownian-results}d illustrates several of these distributions for $T_{\max}$.
Notice that they are unimodal.
Further, the $T_{\max}$ distribution concentrates near $t=0$ as $N$ increases (for fixed $r$) and as $r$ increases (for fixed $N$).
It would be interesting to launch an in-depth mathematical and empirical study of $T_{\max}$ distributions for the motion models we here consider.
Mathematically, one could attempt to prove limit theorems for $T_{\max}$, for instance when $N \to \infty$ and $r \to 0$ with $N r^{2}$ held constant.
Empirically, one could fit known distributions to the data obtained from simulations in an effort to test conjectures about the distribution of $T_{\max}$.
We make one such conjecture here.
For Brownian mobile sensor networks, $T_{\max}$ (when discretized) is distributed approximately geometrically when $N$ and $r$ are large.
The distributions for $(N,r) = (35, 0.16)$ and $(N,r) = (25, 0.20)$ shown in Figure~\ref{fig:Brownian-results}d support this conjecture.

The probability that the sensing balls collectively cover the entire domain $D$ at time $t=0$ (that is, the probability of static coverage) varies widely over $(N,r)$-space (Figure~\ref{fig:Brownian-results}c).
Importantly, even when the probability of static coverage is low (Figure~\ref{fig:Brownian-results}c, lower-left corner), the mobile sensor network nevertheless detects all possible evaders relatively quickly (Figure~\ref{fig:Brownian-results}a, lower-left corner).

As expected, Figure~\ref{fig:Brownian-results}a indicates that $E[T_{\max}]$ decreases as $r$ increases (for fixed $N$) and as $N$ increases (for fixed $r$).
For Brownian mobile sensor networks, $\Var [T_{\max}]$ behaves the same way (Figure~\ref{fig:Brownian-results}b).

\section{Collective motion improves mobile sensor network performance}
\label{sec:deterministic}

Suppose the mobile sensors are capable of self-propulsion.
How should the mobile sensors move in order to optimize performance of the sensor network?
Suppose that each mobile sensor can compute the relative positions of nearby sensors.
Can improved mobile sensor network performance emerge when the mobile sensors use this local position data to locally coordinate?

We answer this question affirmatively by allowing the mobile sensors to use attractive and repulsive forces for local coordination.
In particular, we use the D'Orsogna model~\cite{DOrsogna2006self}, a seminal effort to explain the emergent phenomena that arise when agents move collectively.
We compare this \textquoteleft locally strategic\textquoteright\ model to a \textquoteleft locally oblivious\textquoteright\ alternative wherein the mobile sensors move independently as simple billiard particles.
Both the billiard and D'Orsogna models proceed deterministically after the initial random sensor locations and headings have been chosen, in contrast with the stochastic Brownian motion in the previous section.

\subsection{Billiard motion}

Our domain geometry is as in Section~\ref{subsec:squares}.
We initialize the sensor network by placing each of the $N$ mobile sensors uniformly at random within the unit square $S$ and assigning each a random initial direction of motion $\theta_{k} \in \mathbb{R} / 2 \pi \mathbb{Z}$.
Each mobile sensor moves independently of the others at unit speed and changes direction of motion only when colliding elastically with the virtual boundary $\partial S$.

\subsection{Collective mobile sensor motion}

The D'Orsogna model uses attractive and repulsive potentials to model the dynamics of animal collectives such as fish schools and bird flocks.
Here, we show that if mobile sensors use such potentials to locally coordinate, the emergent collective motion can cause these D'Orsogna sensor networks to outperform billiard sensor networks.

The $N$ mobile sensors $\bm{x}_{k}$ obey the nonlinear system
\begin{subequations}
\label{e:DOrsogna-system}
\begin{align}
\frac{\mathrm{d} \bm{x}_{k}}{\mathrm{d} t} &= \bm{v}_{k}
\\
\mathfrak{m} \frac{\mathrm{d} \bm{v}_{k}}{\mathrm{d} t} &=
(\alpha - \beta \| \bm{v}_{k} \|^{2}) \bm{v}_{k} - \sum_{\substack{m = 1\\ m \neq k}}^{N} g (\| \bm{x}_{m} - \bm{x}_{k} \|) \nabla_{k} U_{m} (\bm{x}_{k}).
\label{e:Newton}
\end{align}
\end{subequations}
The generalized Morse potential
\begin{equation}
\label{e:potential}
U_{m} (\bm{x}_{k}) = C_{r} e^{-\| \bm{x}_{m} - \bm{x}_{k} \| / \ell_{r}} - C_{a} e^{-\| \bm{x}_{m} - \bm{x}_{k} \| / \ell_{a}}
\end{equation}
describes the interaction between sensors $m$ and $k$.
The parameters $\ell_{a}$ and $\ell_{r}$ in Eq.~\eqref{e:potential} quantify the length scales for attraction and repulsion, respectively.
The parameters $C_{a}$ and $C_{r}$ quantify attraction and repulsion magnitudes.
Constant $\mathfrak{m}$ is a sensor's mass.

The function $g$ localizes the interaction force to ensure that sensors $m$ and $k$ only interact if they are near one another.
Our use of the localizer $g$ is consistent with our minimal sensing approach: The mobile D'Orsogna sensors should only have the ability to gather and react to local conditions.
Concretely, $g$ is given by
\begin{equation*}
g(x) =
\begin{cases}
1, &\text{if } x < 2r,
\\
0, &\text{if } x \geqslant 2r.
\end{cases}
\end{equation*}
We note that two sensors interact, affecting each other's motion, if and only if their sensing balls of radius $r$ overlap.

In isolation, mobile D'Orsogna sensor $k$ only feels the force $(\alpha - \beta \| \bm{v}_{k} \|^{2}) \bm{v}_{k}$ and will therefore approach the asymptotic speed $\sqrt{\alpha / \beta }$.
In equilibrium, speeds of the mobile D'Orsogna sensors should statistically fluctuate about this asymptotic value.

\subsection{Parameter selection and initial data for model comparison}

The mobile billiard sensors move at unit speed.
We have set $\alpha = \beta = 1$ in Eq.~\eqref{e:Newton} so that the equilibrium speeds of the mobile D'Orsogna sensors fluctuate about one.
This choice matches the characteristic speeds of the two models, allowing for a meaningful sensor network performance comparison.
We initialize the billiard and D'Orsogna sensor networks in the same way: 
We place each of the $N$ mobile sensors uniformly at random within the unit square $S$ and we assign each a random initial direction of motion $\theta_{k} \in \mathbb{R} / 2 \pi \mathbb{Z}$.

We have set the attraction and repulsion length scales in the D'Orsogna potential~\eqref{e:potential} to $\ell_{a} = 1$ and $\ell_{r} = 0.1$, respectively.
The scale separation here models collective animal behavior, namely short-range repulsion, and attraction over longer distances.
We have set the attraction and repulsion magnitude parameters to $C_{a} = 0.45$ and $C_{r} = 0.5$, respectively.
The mobile D'Orsogna sensors have mass $\mathfrak{m} = 1$.

\subsection{D'Orsogna sensor networks outperform billiard sensor networks}

We computed detection time $T_{\max}$ distributions for both the D'Orsogna sensor network and the billiard sensor network simulations.
We did so by performing 500 simulations for each choice of sensing radius $r$ and number of mobile sensors $N$.
Figure~\ref{fig:model-comparison} compares $E[T_{\max}]$ and $\Var [T_{\max}]$ behavior for the two sensor network models.

\begin{figure}[ht]
\centering
\includegraphics[width=15cm]{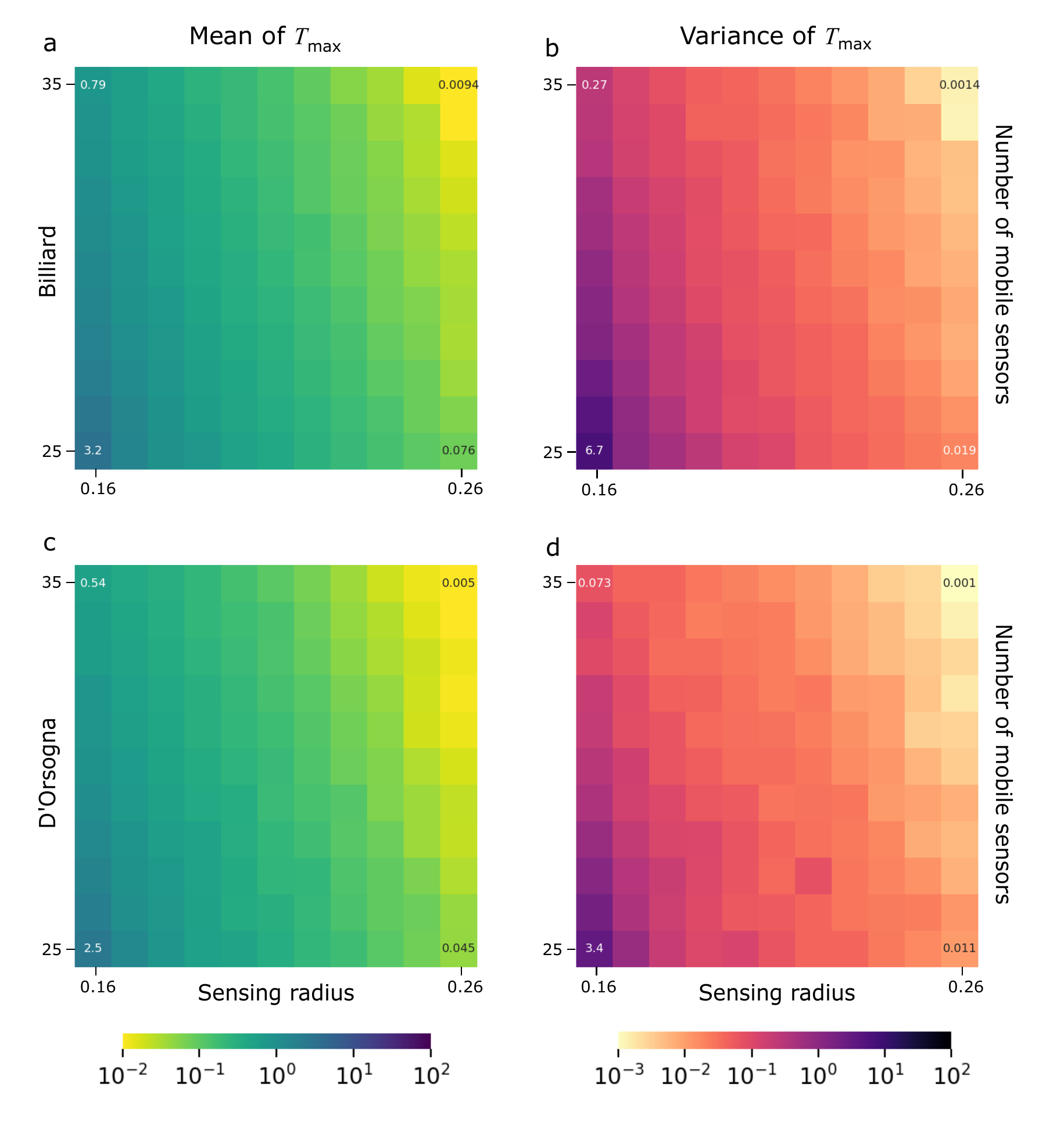}
\caption{
\textbf{Performance comparison of D'Orsogna and billiard mobile sensor networks.}
Heatmaps illustrate detection time statistics as functions of the sensing radius $r$ and the number of mobile sensors $N$.
Left column: Mean detection time.
Right column: Detection time variance.
Top row: Billiard sensor network.
Bottom row: D'Orsogna sensor network.
}
\label{fig:model-comparison}
\end{figure}

Importantly, the D'Orsogna network detects evaders more quickly on average than the billiard network (Figure~\ref{fig:model-comparison}ac).
Quantitatively, averaging
\begin{equation*}
\frac{E[T_{\max}] (\text{billiard}) - E[T_{\max}] (\text{D'Orsogna})}{E[T_{\max}] (\text{billiard})}
\end{equation*}
over all $(N,r)$ pairs, we find that the D'Orsogna network outperforms the billiard network by $23.5\%$.
We speculate that both the repulsive and attractive potentials help the D'Orsogna network outperform the billiard network.
The repulsive potential may cause the typical fraction of the domain covered by the mobile D'Orsogna sensors at a fixed time to eclipse that of the mobile billiard sensors.
The attractive potential may help prevent individual mobile D'Orsogna sensors (or small groups) from disconnecting from the fence component.

As is the case for the Brownian network, $\Var [T_{\max}]$ varies monotonically with $r$ and with $N$ for the billiard network (Figure~\ref{fig:model-comparison}b) and the D'Orsogna network (Figure~\ref{fig:model-comparison}d).

\section{Conclusion}
\label{sec:conclusion}

Minimal sensor networks attempt to solve global sensing problems using local data.
Assuming mobile planar sensors collect local connectivity data, local distance data, and weak rotation information, can we determine whether or not an evasion path exists?
Adams and Carlsson answer this question affirmatively~\cite[Theorem~3]{EvasionPaths}.
Here, we have converted this theorem into explicit algorithms that allow us to compute detection time distributions for both deterministic and stochastic mobile sensor networks.
We have computed and analyzed these distributions for sensor networks with Brownian, billiard, and D'Orsogna motion models.
Comparing the two deterministic mobile networks, the D'Orsogna system outperforms its billiard counterpart.
This suggests that emergent dynamics resulting from collective local motion can positively impact mobile sensor network performance.

We end with open questions that we hope will help inspire future research.
Opportunities exist for both theoretical and computational investigations.

\begin{enumerate}[leftmargin=*, label=\textbf{(\arabic*)}, itemsep=0.5ex]
\item
The expected time to achieve mobile coverage decreases as we increase the number of sensors or as we increase the sensing radius.
Do phase transition phenomena ever arise?
\item 
In the limit as the number of sensors goes to infinity and the sensing radius of each sensor goes to zero, can one describe the asymptotics of the expected time to achieve mobile coverage?
Does the distribution of this detection time, $T_{\max}$, converge in this limit?
If it does, to what limiting distribution?
How do these asymptotics depend on the model of sensor motion?
\item
For deterministic sensor motion models, how do the statistics of $T_{\max}$ depend on the underlying invariant measure and rate of decay of correlations?
\item 
The detection time distributions in Figure~\ref{fig:Brownian-results}d are unimodal.
Is this always the case for Brownian motion?
What happens for different models of motion?
\item 
What \emph{collective} motion models, wherein each mobile sensor can adjust its motion based on the relative positions of its local (but not global) neighbors, are the most effective at decreasing the expected time until mobile coverage is achieved?
\item What can be said about collections of sensors with randomly varying sensing radii?
\item What do the statistics for mobile coverage look like in the power-on model, wherein mobile sensors are allowed to become disconnected but still all remain on?
\item 
How do the statistics for mobile coverage change as the domain varies?
For instance, what happens on a torus, or on other domains that are not simply connected?
\item
How would the presence of hyperbolicity in the dynamics impact mobile sensor network performance?
For instance, what happens if the mobile sensors move as billiard particles and reflect off of convex scatters (the Lorenz gas or Sinai billiard setup)?
\item 
Suppose we have a fixed domain, say a unit cube or a unit ball in $\R^d$.
Let $V>0$ be a fixed volume that will be distributed amongst sensors.
For any integer $N \geqslant 1$, suppose we have $N$ ball-shaped sensors, each of volume $V/N$, which move in the domain according to (say) Brownian motion.
How does the expected time until mobile coverage depend on $N$?
Which value of $N$ (as a function of $V$) minimizes the expected mobile coverage time?
\end{enumerate}

\section{Acknowledgments}

We would like to acknowledge Michael Kerber for showing us the result in \SIADSEDIT{Section~\ref{section:alpha-local}}.
\SIADSEDIT{We would like to thank the anonymous reviewers for their insightful and helpful feedback.}

\section{Funding}

This work has been partially supported by National Science Foundation grant DMS 1816315 (William Ott).

\section{Code availability}

See~\cite{OurGitHubCode} for code and documentation.

\FloatBarrier

\bibliographystyle{siamplain}
\bibliography{ProbabilityEvasion}

\begin{thebibliography}{10}

\bibitem{OurGitHubCode}
\url{https://github.com/elykwilliams/EvasionPaths}.

\bibitem{AdamsThesis}
{\sc H.~Adams}, {\em Evasion paths in mobile sensor networks}, PhD thesis,
  Stanford University, 2013.

\bibitem{EvasionPaths}
{\sc H.~Adams and G.~Carlsson}, {\em Evasion paths in mobile sensor networks},
  International Journal of Robotics Research, 34 (2015), pp.~90--104,
  \url{https://doi.org/10.1177/0278364914548051}.

\bibitem{biasotti2008reeb}
{\sc S.~Biasotti, D.~Giorgi, M.~Spagnuolo, and B.~Falcidieno}, {\em Reeb graphs
  for shape analysis and applications}, Theoretical computer science, 392
  (2008), pp.~5--22.

\bibitem{carlsson2020space}
{\sc G.~Carlsson and B.~Filippenko}, {\em The space of sections of a smooth
  function}, 2020, \url{https://arxiv.org/abs/2006.12023}.

\bibitem{Chin2010detection}
{\sc J.-C. Chin, Y.~Dong, W.-K. Hon, C.-T. Ma, and D.~Yau}, {\em Detection of
  intelligent mobile target in a mobile sensor network}, IEEE/ACM Transactions
  on Networking, 18 (2010), pp.~41--52,
  \url{https://doi.org/10.1109/TNET.2009.2024309}.

\bibitem{chintakunta2013distributed}
{\sc H.~Chintakunta and H.~Krim}, {\em Distributed boundary tracking using
  alpha and {D}elaunay--\v{C}ech shapes}, arXiv preprint arXiv:1302.3982,
  (2013).

\bibitem{Chung2011search}
{\sc T.~Chung, G.~Hollinger, and V.~Isler}, {\em Search and pursuit-evasion in
  mobile robotics {A} survey}, Autonomous Robots, 31 (2011), pp.~299--316,
  \url{https://doi.org/10.1007/s10514-011-9241-4}.

\bibitem{Coordinate-free}
{\sc V.~De~Silva and R.~Ghrist}, {\em Coordinate-free coverage in sensor
  networks with controlled boundaries via homology}, International Journal of
  Robotics Research, 25 (2006), pp.~1205--1222,
  \url{https://doi.org/10.1177/0278364906072252}.

\bibitem{de2007coverage}
{\sc V.~de~Silva and R.~Ghrist}, {\em Coverage in sensor networks via
  persistent homology}, Algebraic and Geometric Topology, 7 (2007),
  pp.~339--358, \url{https://doi.org/10.2140/agt.2007.7.339}.

\bibitem{DeSilva2007homological}
{\sc V.~De~Silva and R.~Ghrist}, {\em Homological sensor networks}, Notices of
  the American Mathematical Society, 54 (2007), pp.~10--17.

\bibitem{DOrsogna2006self}
{\sc M.~D'Orsogna, Y.~Chuang, A.~Bertozzi, and L.~Chayes}, {\em Self-propelled
  particles with soft-core interactions: {P}atterns, stability, and collapse},
  Physical Review Letters, 96 (2006),
  \url{https://doi.org/10.1103/PhysRevLett.96.104302}.

\bibitem{EdelsbrunnerHarer}
{\sc H.~Edelsbrunner and J.~L. Harer}, {\em Computational Topology: An
  Introduction}, American Mathematical Society, Providence, 2010.

\bibitem{edelsbrunner1990simulation}
{\sc H.~Edelsbrunner and E.~P. M{\"u}cke}, {\em Simulation of simplicity: {A}
  technique to cope with degenerate cases in geometric algorithms}, ACM
  Transactions on Graphics, 9 (1990), pp.~66--104,
  \url{https://doi.org/10.1145/77635.77639}.

\bibitem{edelsbrunner1994three}
{\sc H.~Edelsbrunner and E.~P. M{\"u}cke}, {\em Three-dimensional alpha
  shapes}, ACM Transactions On Graphics (TOG), 13 (1994), pp.~43--72.

\bibitem{Elshakhs2024Delaunay}
{\sc Y.~S. Elshakhs, K.~M. Deliparaschos, T.~Charalambous, G.~Oliva, and
  A.~Zolotas}, {\em A comprehensive survey on {D}elaunay triangulation:
  {A}pplications, algorithms, and implementations over {CPU}s, {GPU}s, and
  {FPGA}s}, IEEE Access, 12 (2024), pp.~12562--12585,
  \url{https://doi.org/10.1109/ACCESS.2024.3354709}.

\bibitem{Gamble2012applied}
{\sc J.~Gamble, H.~Chintakunta, and H.~Krim}, {\em Applied topology in static
  and dynamic sensor networks}, 2012,
  \url{https://doi.org/10.1109/SPCOM.2012.6290237}.

\bibitem{ghrist2017positive}
{\sc R.~Ghrist and S.~Krishnan}, {\em Positive {A}lexander duality for pursuit
  and evasion}, SIAM Journal on Applied Algebra and Geometry, 1 (2017),
  pp.~308--327, \url{https://doi.org/10.1137/16M1089083}.

\bibitem{hall1988introduction}
{\sc P.~Hall}, {\em Introduction to the theory of coverage processes}, Wiley
  Series in Probability and Mathematical Statistics: Probability and
  Mathematical Statistics, John Wiley \& Sons, Inc., New York, 1988,
  \url{https://doi.org/10.1016/0167-0115(88)90159-0}.

\bibitem{Hatcher}
{\sc A.~Hatcher}, {\em Algebraic Topology}, Cambridge University Press,
  Cambridge, 2002.

\bibitem{igusa2002higher}
{\sc K.~Igusa}, {\em Higher {F}ranz-{R}eidemeister torsion}, vol.~31 of AMS/IP
  Studies in Advanced Mathematics, American Mathematical Society, Providence,
  RI; International Press, Somerville, MA, 2002,
  \url{https://doi.org/10.1016/s0550-3213(02)00739-3}.

\bibitem{Kellerer1983number}
{\sc A.~M. Kellerer}, {\em On the number of clumps resulting from the overlap
  of randomly placed figures in a plane}, J. Appl. Probab., 20 (1983),
  pp.~126--135, \url{https://doi.org/10.2307/3213726}.

\bibitem{kerber20133d}
{\sc M.~Kerber and H.~Edelsbrunner}, {\em 3d kinetic alpha complexes and their
  implementation}, in 2013 Proceedings of the Fifteenth Workshop on Algorithm
  Engineering and Experiments (ALENEX), SIAM, 2013, pp.~70--77.

\bibitem{distributed}
{\sc D.~Khryashchev, J.~Chu, M.~Vejdemo-Johansson, and P.~Ji}, {\em A
  distributed approach to the evasion problem}, Algorithms (Basel), 13 (2020),
  pp.~Paper No. 149, 13, \url{https://doi.org/10.3390/a13060149}.

\bibitem{La2009flocking}
{\sc H.~La and W.~Sheng}, {\em Flocking control of a mobile sensor network to
  track and observe a moving target}, 2009, pp.~3129--3134,
  \url{https://doi.org/10.1109/ROBOT.2009.5152747}.

\bibitem{Liu2005mobility}
{\sc B.~Liu, P.~Brass, O.~Dousse, P.~Nain, and D.~Towsley}, {\em Mobility
  improves coverage of sensor networks}, 2005, pp.~300--308,
  \url{https://doi.org/10.1145/1062689.1062728}.

\bibitem{Liu2013dynamic}
{\sc B.~Liu, O.~Dousse, P.~Nain, and D.~Towsley}, {\em Dynamic coverage of
  mobile sensor networks}, IEEE Transactions on Parallel and Distributed
  Systems, 24 (2013), pp.~301--311,
  \url{https://doi.org/10.1109/TPDS.2012.141}.

\bibitem{mohargraphs}
{\sc B.~Mohar and C.~Thomassen}, {\em Graphs on surfaces}, Johns Hopkins
  Studies in the Mathematical Sciences, Johns Hopkins University Press,
  Baltimore, MD, 2001.

\bibitem{Moran1974volume}
{\sc P.~A.~P. Moran}, {\em The volume occupied by normally distributed
  spheres}, Acta Math., 133 (1974), pp.~273--286,
  \url{https://doi.org/10.1007/BF02392147}.

\bibitem{Olfati-Saber2006flocking}
{\sc R.~Olfati-Saber}, {\em Flocking for multi-agent dynamic systems:
  {A}lgorithms and theory}, IEEE Transactions on Automatic Control, 51 (2006),
  pp.~401--420, \url{https://doi.org/10.1109/TAC.2005.864190}.

\bibitem{Olfati-Saber2007distributed}
{\sc R.~Olfati-Saber}, {\em Distributed {K}alman filtering for sensor
  networks}, 2007, pp.~5492--5498,
  \url{https://doi.org/10.1109/CDC.2007.4434303}.

\bibitem{Olfati-Saber2007consensus}
{\sc R.~Olfati-Saber, J.~Fax, and R.~Murray}, {\em Consensus and cooperation in
  networked multi-agent systems}, Proceedings of the IEEE, 95 (2007),
  pp.~215--233, \url{https://doi.org/10.1109/JPROC.2006.887293}.

\bibitem{Olfati-Saber2012coupled}
{\sc R.~Olfati-Saber and P.~Jalalkamali}, {\em Coupled distributed estimation
  and control for mobile sensor networks}, IEEE Transactions on Automatic
  Control, 57 (2012), pp.~2609--2614,
  \url{https://doi.org/10.1109/TAC.2012.2190184}.

\bibitem{pascucci2007robust}
{\sc V.~Pascucci, G.~Scorzelli, P.-T. Bremer, and A.~Mascarenhas}, {\em Robust
  on-line computation of {R}eeb graphs: simplicity and speed}, in ACM SIGGRAPH
  2007 papers, 2007, pp.~58--es.

\bibitem{Qu2014finite}
{\sc Y.~Qu, S.~Xu, C.~Song, Q.~Ma, Y.~Chu, and Y.~Zou}, {\em Finite-time
  dynamic coverage for mobile sensor networks in unknown environments using
  neural networks}, Journal of the Franklin Institute, 351 (2014),
  pp.~4838--4849, \url{https://doi.org/10.1016/j.jfranklin.2014.05.011}.

\bibitem{Schroeter1984distribution}
{\sc G.~Schroeter}, {\em Distribution of number of point targets killed and
  higher moments of coverage of area targets}, Naval research logistics
  quarterly, 31 (1984), pp.~373--385,
  \url{https://doi.org/10.1002/nav.3800310304}.

\bibitem{Su2016distributed}
{\sc H.~Su, X.~Chen, M.~Chen, and L.~Wang}, {\em Distributed estimation and
  control for mobile sensor networks with coupling delays}, ISA Transactions,
  64 (2016), pp.~141--150, \url{https://doi.org/10.1016/j.isatra.2016.04.025}.

\bibitem{Su2017distributed}
{\sc H.~Su, Z.~Li, and M.~Chen}, {\em Distributed estimation and control for
  two-target tracking mobile sensor networks}, Journal of the Franklin
  Institute, 354 (2017), pp.~2994--3007,
  \url{https://doi.org/10.1016/j.jfranklin.2017.01.033}.

\bibitem{vijayan1982planarity}
{\sc G.~Vijayan and A.~Wigderson}, {\em Planarity of edge ordered graphs},
  Technical Report 307, Department of Electrical Engineering and Computer
  Science, Princeton University,  (1982).

\bibitem{Yuan2019Temnothorax}
{\sc W.~Yuan, N.~Ganganath, C.-T. Cheng, Q.~Guo, and F.~Lau}, {\em Temnothorax
  albipennis migration inspired semi-flocking control for mobile sensor
  networks}, Chaos, 29 (2019), \url{https://doi.org/10.1063/1.5093073}.

\end{thebibliography}

\appendix

\section{Pseudocode}

This appendix contains pseudocode for algorithms referenced throughout the paper.
Algorithm~\ref{alg:intruder-detection-adaptive} shows how the timestepping can be performed in an adaptive manner to ensure that only one atomic combinatorial change happens at a time.
Algorithm~\ref{alg:UpdateLabel} shows a case-based algorithm, to be contrasted with the non-case-based (and much shorter) Algorithm~\ref{alg:UpdateLabel-v2}; cf.\ Section~\ref{ssec:non-state}.
Finally, Algorithm~\ref{alg:UpdateLabel-v3} shows how the assumption that the alpha complex remains connected may be relaxed while updating labels for the boundary cycles in the power-down model.

\begin{figure}[ht]
\begin{algorithm}[H]\label{alg:intruder-detection-adaptive}
\SetAlgoLined
\SetKwInOut{Input}{Input}
\SetKwInOut{Output}{Output}
\Input{Sensor positions $X_{n}$, alpha complex $A_{n}$, the boundary cycles $B_{n}$, boundary cycle labelling $L_{n}\colon B_{n} \rightarrow \{0, 1\}$, current time $t$, timestep size $\Delta t$, sensing radius $R$}
\Output{Sensor positions $X_{n+1}$, alpha complex $A_{n+1}$, the boundary cycles $B_{n+1}$, boundary cycle labelling $L_{n+1}\colon B_{n+1} \rightarrow \{0, 1\}$}
\Begin{
$X_{n+1}^{*} \leftarrow UpdatePositions(X_{n}, dt)$\;
$A_{n+1}^{*} \leftarrow AlphaComplex(X_{n}^{*}, R)$\;
$B_{n+1}^{*} \leftarrow ComputeBoundaryCycles(A_{n}^{*}, X_{n}^{*})$\;
Remove boundary cycle corresponding to the outside of fence from $B_{n+1}^{*}$\;
\eIf{Atomic Transition occurred}{
    $L_{n+1} \leftarrow LabelUpdate(B_{n}, B_{n}^{*}, A_{n}^{*}, L_{n})$ via Algorithm~\ref{alg:UpdateLabel-v2} or~\ref{alg:UpdateLabel-v3}\;
    \Return{$X_{n+1}^{*}, A_{n+1}^{*}, B_{n+1}^{*}, L_{n+1}$}
}{
    $X_{n+\frac{1}{2}}, A_{n+\frac{1}{2}}, B_{n+\frac{1}{2}}, L_{n+\frac{1}{2}} \leftarrow AdaptiveTimeStep(\frac{\Delta t}{2}, t, R, X_{n}, A_{n}, B_{n}, L_{n})$\;
    $X_{n+1}, A_{n+1}, B_{n+1}, L_{n+1} \leftarrow AdaptiveTimeStep(\frac{\Delta t}{2}, t+\frac{\Delta t}{2}, R, X_{n+\frac{1}{2}}, A_{n+\frac{1}{2}}, B_{n+\frac{1}{2}}, L_{n+\frac{1}{2}})$\;
    \Return{ $X_{n+1}, A_{n+1}, B_{n+1}, L_{n+1} $}
}
}
\caption{Single Adaptive Time-Step}
\end{algorithm}
\end{figure}

\begin{figure}[ht]
\resizebox{.95\columnwidth}{!}{
\begin{algorithm}[H]\label{alg:UpdateLabel}
\SetAlgoLined
\SetKwInOut{Input}{Input}
\SetKwInOut{Output}{Output}
\Input{Alpha complex $A_{n+1}$, the boundary cycles $B_{n}$ and $B_{n+1}$, boundary cycle labelling $L_{n}\colon B_{n} \rightarrow \{0, 1\}$, fence sensors $F$ }
\Output{The mapping $L_{n+1}\colon B_{n+1} \rightarrow \{0, 1\}$, indicating whether each boundary cycle may contain an intruder}

\Begin{
\small{
    \For{$b \in B_{n} \cap B_{n+1}$}{
        $L_{n+1} \leftarrow L_{n}$\;
    }
    \Switch{Atomic transition}{
        \Case{1-Simplex is added}{
            $x \leftarrow$ only cycle in $B_{n}\setminus B_{n+1}$ \;
            \For{$b \in B_{n+1}\setminus B_{n}$}{
                $L_{n+1}[b] \leftarrow L_{n+1}[x]$\;
            }
        }
        \Case{1-Simplex is removed}{
            $b \leftarrow$ only cycle in $B_{n+1}\setminus B_{n}$ \;
            $L_{n+1}[b] \leftarrow \bigvee\limits_{x \in  B_{n}\setminus B_{n+1}} L_{n}[x]$\;
            
        }
        \Case{2-Simplex is added}{
            $b \leftarrow$ boundary cycle associated with new simplex\;
            $L_{n+1}[b] \leftarrow$ \False \;
        }
        \Case{2-Simplex is removed}{
            \Continue \;
        }
        \Case{Simplex pair added}{
            $x \leftarrow$ only cycle in $B_{n}\setminus B_{n+1}$ \;
            \For{$b \in B_{n+1}\setminus B_{n}$}{
                \eIf{Nodes of $b$ form a 2-Simplex}{
                    $L_{n+1}[b] \leftarrow \False$\;
                }{
                    $L_{n+1}[b] \leftarrow L_{n+1}[x]$\;
                }
            }
        }
        \Case{Simplex pair removed}{
            $b \leftarrow$ only cycle in $B_{n+1}\setminus B_{n}$ \;
            $L_{n+1}[b] \leftarrow \bigvee\limits_{x \in  B_{n}\setminus B_{n+1}} L_{n+1}[x]$\;
        }
        \Case{Delaunay edge flip}{
            \For{$b \in B_{n+1}\setminus B_{n}$}{
                $L_{n+1}[b] \leftarrow \False$\;
            }
        }
    }
    \For{$b \in B_{n}\setminus B_{n+1}$}{
            $L_{n+1}[b] \leftarrow NULL$\;
        }
    \Return $L_{n+1}$\;
}
}
\caption{Update Labelling}
\end{algorithm}
}
\end{figure}

\begin{figure}[ht]
\begin{algorithm}[H]\label{alg:UpdateLabel-v3}
\SetAlgoLined
\SetKwInOut{Input}{Input}
\SetKwInOut{Output}{Output}
\Input{Alpha complex $A_{n+1}$, the boundary cycles $B_{n}$ and $B_{n+1}$, boundary cycle labelling $L_{n}\colon B_{n} \rightarrow \{0, 1\}$, fence sensors $F$ }
\Output{The mapping $L_{n+1}\colon B_{n+1} \rightarrow \{0, 1\}$, indicating whether each boundary cycle may contain an intruder}

\Begin{
    \small{
    \For{$b \in \{y \in B_{n+1} \cap B_{n} \mid \text{nodes of y are connected to fence}\}$}{
                $L_{n+1}[b] = L_{n}[b]$;
    }
    }
    \Switch{Atomic transition}{
        \Case{1-simplex added causing re-connection}{
            $V = \{ b \in B_{n}\setminus B_{n+1} \mid b \text{ is connected to } F \}$\;
            \For{$b \in B_{n+1}$}{
                \If{ $ b \text{ is connected to } F \text{ but has no label in } L_{n} $}{
                    $L_{n+1}[b] \leftarrow \bigvee\limits_{x \in V} L_{n}[x]$\;
                }\ElseIf{$b \text{ is connected to } F$ }{
                    $L_{n+1}[b] \leftarrow L_{n}[b]$\;
                }
            }
        }
        \Case{1-simplex removed causing disconnection}{
            $V = \{ b \in B_{n+1} \mid \; b \text{ is not connected to } F \text{ but has label in } L_{n}\}$\;
            \For{$b \in B_{n+1}\setminus B_{n}$}{
                \If{$b \text{ is connected to } F$}{ 
                    $L_{n+1}[b] \leftarrow \bigvee\limits_{x \in V} L_{n}[x] \vee L_{n}[b]$\;
                }{}
            }

        }
        \Other{
            \For{$b \in B_{n+1} \setminus B_{n}$}{
                \If{$b \text{ is connected to } F$}{
                    $L_{n+1}[b] \leftarrow \bigvee\limits_{x \in  B_{n}\setminus B_{n+1}} L_{n}[x]$\;
                }{}
            }
        }   
    }
    \For{$b \in B_{n+1}$}{
        \If{Nodes of $b$ form a 2-Simplex and are connected to $F$}{
            $L_{n+1}[b] \leftarrow \mathbf{False}$\;
        }{}
    }
    
    \Return $L_{n+1}$\;
}
 \caption{Update Labelling with Disconnected Components}
\end{algorithm}
\end{figure}

\FloatBarrier

\end{document}